\documentclass[a4paper,12pt,reqno]{amsart}

\pdfoutput=1
\calclayout

\numberwithin{equation}{section}

\usepackage[T1]{fontenc}
\usepackage[english]{babel}
\usepackage{amsthm}
\usepackage{amsfonts}
\usepackage{amsmath}
\usepackage{amssymb}
\usepackage{comment}
\usepackage{color}
\usepackage{booktabs} 

\usepackage{mathtools}
\usepackage{enumerate}

\usepackage{graphicx}
\usepackage{enumitem}
\usepackage{tabto}
\usepackage{array}
\usepackage{longtable}
\newcolumntype{C}[1]{>{\centering\arraybackslash}p{#1}}

 \usepackage{hyperref}

 \DeclareFontFamily{U}{wncy}{}
 \DeclareFontShape{U}{wncy}{m}{n}{<->wncyr10}{}
 \DeclareSymbolFont{mcy}{U}{wncy}{m}{n}
 \DeclareMathSymbol{\Sh}{\mathord}{mcy}{"58} 
\DeclareMathOperator{\Res}{Res}
\DeclareMathOperator{\Vol}{Vol}

\DeclareMathOperator{\disc}{Disc}

\DeclareMathOperator{\Gal}{Gal}

\DeclareMathOperator{\Sel}{Sel}

\DeclareMathOperator{\D}{\mathcal{D}}

\DeclareMathOperator{\Disc}{Disc}

\DeclareMathOperator{\lcm}{lcm}

\DeclareMathOperator{\mult}{mult}
\DeclareMathOperator{\tors}{tors}

\DeclareMathOperator{\rX}{X}
\DeclareMathOperator{\rY}{Y}

\DeclareMathOperator{\rf}{f}

\DeclareMathOperator{\rT}{T}

\newcommand{\Q}{\mathbb{Q}}

\newcommand{\R}{\mathbb{R}}
\newcommand{\PP}{\mathbb{P}}
\newcommand{\Z}{\mathbb{Z}}
\newcommand{\F}{\mathbb{F}}

\newcommand{\leg}[2]{\left(\frac{#1}{#2}\right)}

\newcommand{\cQ}{\mathcal{Q}}

\newcommand{\cG}{\mathcal{G}}

\newcommand{\abs}[1]{\left|#1\right|}

\newcommand{\ZZ}{\mathbb{Z}}
\newcommand{\QQ}{\mathbb{Q}}

\newcommand{\EE}{\mathbb{E}}

\newcommand{\cT}{\mathcal{T}}

\newcommand{\cS}{\mathcal{S}}
\newcommand{\cR}{\mathcal{R}}
\newcommand{\cU}{\mathcal{U}}

\newcommand{\cE}{\mathcal{E}}
\newcommand{\cM}{\mathcal{M}}

\newcommand{\cA}{\mathcal{A}}

\newcommand{\cF}{\mathcal{F}}

\renewcommand{\epsilon}{\varepsilon}
\renewcommand{\leq}{\leqslant}
\renewcommand{\geq}{\geqslant}

\newtheorem{theorem}{Theorem}[section]
\newtheorem{proposition}[theorem]{Proposition}

\newtheorem{lemma}[theorem]{Lemma}

\theoremstyle{definition}

\newtheorem{definition}[theorem]{Definition}

\theoremstyle{remark}

\begin{document}

\author{Stephanie Chan}
\author{Matteo Verzobio}
\address{Institute of Science and Technology Austria\\
Am Campus 1\\
3400 Klosterneuburg\\
Austria}
\email{stephanie.chan@ist.ac.at}
\email{matteo.verzobio@gmail.com}
\title[Selmer ratios of families of elliptic curves]{Selmer groups of families of elliptic curves with an $\ell$-isogeny}
\date{\today}

\begin{abstract}
For certain families of elliptic curves admitting a rational isogeny of prime degree $\ell$, we establish a central limit theorem for the Tamagawa ratio and derive bounds on its average value. By using the Tamagawa ratio to bound the size of the $\ell$-isogeny Selmer group from below, we show that for $\ell \in\{ 2, 3, 5, 7, 13\}$, there exist elliptic curves with arbitrarily large $\ell$-Selmer groups.
\end{abstract}


\maketitle

\setcounter{tocdepth}{1}
\tableofcontents

\section{Introduction}
It is conjectured that the average size of the $n$-Selmer groups of all elliptic curves over $\Q$ is finite~\cite{poonen2012random}. This has been proven for $n \in \{2,3,4,5\}$ by Bhargava and Shankar~\cite{MR3272925,MR3275847,bhargava4,bhargava2013average}. Nevertheless, the size of the $n$-Selmer group is not necessarily uniformly bounded. In fact, there exist families, of density $0$ among all elliptic curves, for which the $n$-Selmer group size is unbounded.

The first result in this direction is due to Klagsbrun and Lemke Oliver~\cite{klagsbrun2014distribution,KLOtwist}, who proved that the $2$-Selmer group is unbounded in the family of elliptic curves with a rational $2$-torsion point, as well as in certain quadratic twist families with partial $2$-torsion. Kane and Thorne~\cite{KaneThorne} considered a quartic twist family of elliptic curves with partial $2$-torsion. Families with unbounded $3$-Selmer groups were studied by Alpöge, Bhargava, and Shnidman~\cite[Theorem 1.11]{ABS}, and by Chan~\cite{Chan3sel}. A common feature in these results is that the elliptic curves admit a rational $n$-isogeny, and the ratio of the sizes of the Selmer groups associated with the isogeny and its dual plays a key role.

In this paper, we demonstrate the unboundedness of the $n$-Selmer group in new families for $n \in \{2,3,5, 7, 13\}$. Our results also encompass all families of elliptic curves with a prescribed torsion subgroup.

Let $\cS$ be a family of elliptic curves over $\Q$ such that each $E \in \cS$ admits a cyclic rational isogeny $\phi: E \to E'$ of prime degree~$\ell$. We study the relative sizes of the Selmer groups $\Sel_{\phi}(E/\Q)$ and $\Sel_{\hat{\phi}}(E'/\Q)$ associated to the isogeny $\phi$ and its dual $\hat{\phi}$, respectively. In particular, we investigate how the difference
\begin{equation}\label{eq:selmerratiorankdef}
r_{\phi}(E) \coloneqq \dim_{\F_\ell} \Sel_{\phi}(E/\Q) - \dim_{\F_\ell} \Sel_{\hat{\phi}}(E'/\Q)
\end{equation}
varies within $\cS$ when the elliptic curves are ordered by naive height.
Let $\cS(N)$ denote the set of elliptic curves in $\cS$ with naive height at most $N$. Theorem~\ref{theorem:mainfam} is an application of our main technical theorem, Theorem~\ref{theorem:mainhom}.
\begin{theorem}\label{theorem:mainfam}
Let $\cS$ be one of the following families of elliptic curves over $\Q$ with a choice of prime $\ell$:
\begin{itemize}
 \item the family of all elliptic curves with $E_{\tors}(\Q) \cong T$, for a non-trivial torsion subgroup $T$ such that $\ell \mid \abs{T}$;
 \item the family of all elliptic curves admitting a rational cyclic isogeny of degree $4$, with $\ell=2$;
 \item an infinite family of elliptic curves in distinct $\overline{\QQ}$-isomorphism classes as described in Section~\ref{subsec:7}, with $\ell = 7$;
 \item an infinite family of elliptic curves as described in Section~\ref{subsec:13}, with $\ell = 13$.
\end{itemize}
Then there exists a choice of rational degree $\ell$ isogeny $\phi$ for each $E\in\cS$ such that all of the following hold:
\begin{enumerate}
 \item There exist constants $\mu$ and $\sigma$ such that
 \[
 \left\{ \frac{r_{\phi}(E) - \mu \log\log N}{\sigma \sqrt{\log\log N}} : E \in \cS(N) \right\}
 \]
 converges in distribution to a standard Gaussian as $N \to \infty$. 
\label{MEK}
 \item Given any $k>0$, there exists a constant $\rho(k)$, such that
 \[
 \sum_{E \in \cS(N)} \ell^{k\cdot r_{\phi}(E)} \gg_k |\cS(N)|(\log N)^{\rho(k)}
 \]
 for all $N \geq 2$. \label{Mav} 
 \item Given any $A>0$, there exist $\delta(A)>0$ such that
 \[
 \left| \left\{ E \in \cS(N) : r_{\phi}(E) \geq A \log\log N \right\} \right| \gg_A |\cS(N)| (\log N)^{-\delta(A)}
 \]
 for all $N \geq 2$. \label{Mtail}
\end{enumerate}
In particular, other than the cases $T\in\{\Z/2\Z\times\Z/2m\Z:m=1,3,4\}$ with $\ell=2$, we have $\rho(1)>0$ for $\cS$. In every case, $\rho(2)>0$ for $\cS$.
The constants $\mu$, $\sigma$, and $\rho(1)$ for $\cS$ are given in Tables~\ref{table:oddprime},~\ref{table:twounique},~\ref{table:twochoice},~\ref{table:4cyclic}, and~\ref{todd}.
\end{theorem}

The first part of the theorem is a central limit theorem-type result, as in Klagsbrun and Lemke Oliver~\cite{klagsbrun2014distribution,KLOtwist}. For similar results in the context of Selmer groups, see also~\cite[Theorem~6.2]{chan2019ranks},~\cite{xiong2008distribution}. The second statement follows from methods used in averaging multiplicative functions over integer sequences~\cite{Wolke}. The third is an application of the Paley--Zygmund inequality to $\ell^{r_{\phi}(E)}$, using~\eqref{Mav} with a sufficiently large choice of $k$.

It follows immediately from Theorem~\ref{theorem:mainfam}\eqref{Mtail} that $r_{\phi}(E)$ can become arbitrarily large in any of the families we consider. From the exact sequence~\cite[Lemma 6.1]{MR2021618}
\[0\rightarrow \frac{E'(\Q)[\hat{\phi}]}{\phi(E(\Q)[\ell])}\rightarrow
\Sel_{\phi}(E/\Q)\rightarrow
\Sel_{\ell}(E/\Q)\rightarrow
\Sel_{\hat{\phi}}(E'/\Q)\rightarrow \frac{\Sh(E'/\Q)[\hat{\phi}]}{\phi(\Sh(E/\Q)[\ell])}\rightarrow 0,
\]
we obtain the bound
\[|\Sel_{\ell}(E/\Q)|\geq \frac{|\Sel_{\phi}(E/\Q)|}{|E'(\Q)[\hat{\phi}]/\phi(E(\Q)[\ell])|}\gg_\ell |\Sel_{\phi}(E/\Q)|\geq \ell^{r_{\phi}(E)},\]
which shows that $|\Sel_{\ell}(E/\Q)|$ can also be arbitrarily large.

To establish our main theorem, it is essential to have a suitable parametrisation of the elliptic curves in $\cS$. The families we consider fall into two types. The first consists of elliptic curves of the form $y^2 = x^3 + f(t)x + g(t)$, where $f, g \in \Z[t]$ are fixed and $t$ varies over $\Q$. We can alternatively view them as the non-singular fibres of a given elliptic surface over an affine line, or over $\PP^1$ after completion. The second consists of all quadratic twists of the first. In either case, we may put the curve into short Weierstrass form $E_{a,b}: y^2 = x^3 + A(a,b)x + B(a,b)$, where $A, B$ are weighted homogeneous polynomials with integer coefficients, and $(a,b)\in\Z^2$ is a representative of a weighted projective point.
For parametrisations of elliptic curves with a given torsion subgroup, see~\cite{barrios2022minimal} and~\cite{harron2017counting}; for parametrisations of elliptic curves admitting a cyclic $n$-isogeny, see~\cite[Table~7]{Maier}.

We rely on a result of Cassels~\cite{Cassels8}:
\[
\frac{|\Sel_{\phi}(E/\Q)|}{|\Sel_{\hat{\phi}}(E'/\Q)|} \asymp_{\ell} \prod_p \frac{c_p(E')}{c_p(E)},
\]
which relates the Selmer ratio to the ratio of Tamagawa numbers $c_p$, given an isogeny $\phi:E\rightarrow E'$. Based on work by Dokchitser and Dokchitser~\cite{DokchitserDokchitser}, we can determine the value of $c_p(E')/c_p(E)\in\{\ell^{-1}, 1, \ell\}$ explicitly from the reduction of $E$ modulo $p$.

A key observation is that $\prod_p c_p(E_{a,b}') / c_p(E_{a,b})$ behaves like the product of multiplicative functions $\ell^{\omega(\cdot)}$ and $\ell^{-\omega(\cdot)}$ evaluated at two independent sequences indexed by $(a,b)$. Taking logarithms, these resemble sums of independent additive functions. This explains the similarity between~\eqref{MEK} and the Erdős--Kac theorem, and between~\eqref{Mav} and average bounds for multiplicative functions as in~\cite{Wolke}.

To prove Theorem~\ref{theorem:mainfam}\eqref{MEK} and~\eqref{Mav}, we show that $(a,b)$ are equidistributed in residue classes as the naive height of $E_{a,b}$ increases. We select families for which we can estimate $|\cS(N)|$ with a good enough error term and incorporate the relevant congruence conditions. For elliptic curves with given torsion subgroups, the order of $|\cS(N)|$ was found by Harron and Snowden~\cite{harron2017counting}, and later refined to an asymptotic by Cullinan, Kenney, and Voight~\cite{cullinan}. For curves admitting cyclic $n$-isogenies, asymptotics were established by Pizzo--Pomerance--Voight for $n=3$~\cite{3iso}, Pomerance--Schaefer for $n=4$~\cite{pomerance2021elliptic}, Arango-Piñeros--Han--Padurariu--Park for $n=5$~\cite{5iso}, and Molnar--Voight for $n=7$~\cite{MolnarVoight}. Boggess and Sankar~\cite{boggess2024counting} obtained order-of-magnitude results for more general values of~$n$.

Once we have the required equidistribution result, we can prove~\eqref{MEK} via the method of moments and the central limit theorem. For~\eqref{Mav}, we adapt the proof of~\cite[Theorem 1.13]{MR4870247}, which builds on~\cite[Satz 2]{Wolke}. 

In Section~\ref{sec:setup}, we describe the elliptic curve families under consideration. In Section~\ref{sec:tamagawa}, we analyse the variation of $c_p(E')/c_p(E)$ within each family. Section~\ref{sec:lattice} is dedicated to counting $E_{a,b}$ with congruence constraints. In Section~\ref{sec:central}, we state and prove our main technical theorem. Finally, in Section~\ref{sec:fam}, we apply this result to the families introduced in Theorem~\ref{theorem:mainfam}.

\subsection*{Acknowledgments}
The authors thank Barinder Banwait, Oana Padurariu, Sun Woo Park, and Efthymios Sofos for helpful discussions. 

The second author was supported by the European Union’s Horizon 2020 research and 
	innovation program under the Marie Skłodowska-Curie Grant Agreement No. 
	101034413.

\section{The set up}\label{sec:setup}
\subsection{Heights and parametrisations}
The goal of this section is to set up our counting problem. In particular, we will define the families with which we will work.

Given $A,B\in \Z$, define
\[
m(A,B)\coloneqq \max\{d\in\Z:d^{12}\mid \gcd(A^3,B^2)\}.
\]
Given an elliptic curve $E$ over $\Q$ in short Weierstrass model 
$y^2=x^3+Ax+B$ for $A,B\in\Z$,
the naive height of $E$ is defined to be
\[
H(E)\coloneqq\frac{\max\{4|A|^3,27B^2\}}{m(A,B)^{12}}.
\]
Also define 
\[
H_{0}(E)\coloneqq \max\{4|A|^3,27B^2\}.
\]
In particular, the definition of $H(E)$ does not depend on the $\Q$-isomorphism class of $E$, but $H_0$ does. Also note that $H\leq H_0$. 
Given any elliptic curve $E/\Q$, we denote by $E^d$ the quadratic twists of $E$ by $d$.

Fix $\upsilon\in\{1,2\}$.
Let $f,g\in\Z[t]$ be polynomials such that
\begin{gather}
 f\text{ and }g\text{ have no common real roots, and}\label{eq:commonrootfg}
 \\
 \max\left\{\frac{1}{2}\deg f,\frac{1}{3}\deg g\right\}=\frac{2m}{\upsilon\tau},\text{ where }m,\tau\in\Z_{\geq 1}\text{ and either }m=1\text{ or }\upsilon\tau=1.\label{eq:degfg}
\end{gather}
For brevity, we henceforth view $f,g,m,\tau,\upsilon$ as fixed, and suppress the dependency throughout the notation.

Let
\[
\cE:y^2=x^3+f(t)x+g(t)
\]
be an elliptic curve over $\Q(t)$, so $4f^3+27g^2\neq 0$ in $\Q(t)$.
We will consider the specialisation
\[
\cE_t:y^2=x^3+f(t)x+g(t),\qquad t\in\Q.
\]
We are interested in the families of elliptic curves
\begin{align*} 
 \cF&\coloneqq\left\{\cE_t: t\in\Q,\ \Disc(\cE_t)\neq 0\right\},\\
 \cG&\coloneqq\left\{\cE_t^d: t\in\Q,\ d\in\Z\text{ squarefree},\ \Disc(\cE_t^d)\neq 0\right\},
\end{align*}
which we view as multisets indexed by the set of $t$ and the set of $(t,d)$ respectively. We treat $\cE_{t}$ and $\cE_{t'}$ as distinct elements of $\cF$ as long as $t\neq t'$.
We study these families ordered by naive height, so define
\begin{equation}\label{def:cFG}
 \begin{aligned}
 \cF(N)&\coloneqq\{E\in\cF: H(E)\leq N\},\\
 \cG(N)&\coloneqq\{E\in\cG: H(E)\leq N\}.
\end{aligned} 
\end{equation}

The rational weighted projective line $\PP_{(d_1,d_2)}^1(\Q)$ with weights $d_1,d_2\in\Z_{\geq 1}$ is the set of equivalence classes of pairs $(a,b)\in \Q^2\setminus\{(0,0)\}$, where $(a,b)$ and $(a',b')$ are equivalent if there exists $\lambda \in \Q^\times$ such that $(a',b')=(\lambda^{d_1} a,\lambda^{d_2} b)$. We write $[a:b]$ for the corresponding point when $(d_1,d_2)$ is fixed.

Define $A,B\in\Z[x,y]$ such that 
\begin{equation}\label{eq:defAB}
A(x,y)\coloneqq y^{2\varsigma}f(x/y^{\tau})\quad\text{and} \quad B(x,y)\coloneqq y^{3\varsigma}g(x/y^{\tau}),\text{ where } \varsigma\coloneqq\begin{cases}
 \hfil 2m&\text{if }\upsilon=1,\\
 \hfil 1&\text{if }\upsilon=2.
\end{cases}
\end{equation}
By construction, $A(x,y)$ and $B(x,y)$ are weighted homogeneous polynomials with weights $\tau,1$. For any $a,b\in\Q$ such that $4A(a,b)^3+27B(a,b)^2\neq 0$, define an elliptic curve
\[
E_{a,b}:y^2=x^3+ A(a,b)x+B(a,b).
\]

The following lemma shows that we can alternatively parametrise $\cF$ and $\cG$ by $\PP_{(\tau,1)}^1(\Q)$ and $\PP_{(2\tau,2)}^1(\Q)$ respectively.
\begin{lemma}\label{lemma:maincorr}
Let $\upsilon$, $\tau$, and $\varsigma$ as in~\eqref{eq:degfg},~\eqref{eq:defAB}. Let $a,b\in\Q$.
\begin{enumerate}
\item
Given any $\alpha\in\PP^1_{(\upsilon\tau,\upsilon)}(\Q)$, the elliptic curves $E_{a,b}$ for $[a:b]=\alpha$ all lie in the same $\Q$-isomorphism class.
 \item Suppose $\varsigma$ is even and $\upsilon=1$. Then 
 $\cE_t\cong E_{a,b}$
 over $\Q$ whenever $[a:b]=[t:1]$ in $\PP^1_{(\tau,1)}(\Q)$.
\item Suppose $\varsigma=1$ and $\upsilon=2$. Then 
 $\cE_t^d\cong E_{a,b}$
 over $\Q$ whenever $[a:b]=[d^{\tau}t:d]$ in $\PP^1_{(2\tau,2)}(\Q)$.
\end{enumerate} 
\end{lemma}

\begin{proof}
For the first part, it suffices to show that $E_{a,b}\cong E_{\lambda^{\upsilon\tau}a,\lambda^{\upsilon}b}$ for any $a,b\in\Q$, $(a,b)\neq (0,0)$, and $\lambda\in\Q^{\times}$.
This is immediate upon verifying
\[A(\lambda^{\upsilon\tau}a,\lambda^{\upsilon}b)=\lambda^{2\varsigma\upsilon}A(a,b)=\lambda^{4m}A(a,b)\quad\text{ and }\quad B(\lambda^{\upsilon\tau}a,\lambda^{\upsilon}b)=\lambda^{6m}B(a,b).\]
The second part is a corollary of the first because $A(t,1)=f(t)$ and $B(t,1)=g(t)$.
For the last part, $\varsigma=1$ implies that
$A(d^{\tau}t,d)=d^{2}f(t)$ and $B(d^{\tau}t,d)=d^3g(t)$.
Hence, $\cE_t^d\cong E_{d^{\tau}t,d}$ and the claim follows again from the first part.
\end{proof}

For $\upsilon\in\{1,2\}$ and $\tau\in\Z_{\geq 1}$, define 
\[
\cT_{\upsilon,\tau}\coloneqq \left\{(a,b)\in\Z^2_{\neq (0,0)}:p^{\upsilon\tau}\nmid \gcd(a,b^{\tau})\text{ for any prime }p\right\},
\]
which is, up to sign, a set of representatives of $\PP^1_{(\upsilon\tau,\upsilon)}(\Q)$.
In particular $\cT_{1,1}=\{(a,b)\in\Z^2_{\neq (0,0)}:\gcd(a,b)=1\}$.
For any $\tau\in\Z_{\geq 1}$, up to fixing the sign of $b$, the set $\cT_{1,\tau}$ is in bijection with $\Q\cup\{\infty\}$ under the map $(a,b)\mapsto a/b^{\tau}$.

In light of Lemma~\ref{lemma:maincorr}, and the definition of $H_0$ and $H$, consider the set
\begin{equation}\label{eq:a1b1}
 \cA^{\delta}_{\upsilon}(N)\coloneqq\left\{(a,b)\in\cT_{\upsilon,\tau}: \begin{array}{l}
 \max\{4|A(a,b)|^3,27B(a,b)^2\}\leq Ne^{12\delta},\\
 \text{where }e=m(A(a,b),B(a,b))
\end{array} \right\}.
\end{equation}
Up to the sign of $(a,b)$, the condition $\Delta(E_{a,b})\neq 0$, and the point $(1,0)$, $\cA^1_{1}(N)$ corresponds to $\cF(N)$ when $\varsigma$ is even, and $\cA^1_{2}(N)$ corresponds to the multiset $\cG(N)$ when $\varsigma=1$. Since $H\leq H_0$, we have $\cA^0_{1}(N)\subseteq\cA^1_{1}(N)$ for every $N$.
Define
\begin{equation}\label{def:cF0}
 \cF_0(N)\coloneqq \{\cE_t:\Disc(\cE_t)\neq 0,\ t=a/b^{\tau}\text{ for some }(a,b)\in\cA^0_{1}(N)\}\subseteq \cF(N).
\end{equation}

We will study congruence conditions on $\cA^{\delta}_{\upsilon}(N)$. Let $q\in\Z_{\geq 1}$. Given $\alpha\in \PP^1(\Z/q\Z)$, define when $\tau=\upsilon=1$,
\begin{equation}\label{eq:tq}
 \cA^{\delta}_{1}(N,\alpha,q)\coloneqq\left\{(a,b)\in \cA^{\delta}_{1}(N): 
[a:b]\equiv \alpha\bmod q\right\}.
\end{equation}
Given $(a_1,b_1)\in(\Z/q\Z)^2$, define
\begin{equation}\label{eq:a1b1q}
 \cA^1_{\upsilon}(N,(a_1,b_1),q)\coloneqq\left\{(a,b)\in\cA^1_{\upsilon}(N): (a,b)\equiv (a_1,b_1)\bmod q \right\}.
\end{equation}
\subsection{Fundamental domains}
We record some properties that will be useful later. 
 Given any polynomials $f,g\in\Q[\mathbf{t}]$ (in arbitrarily many variables),
denote by $\mult_g(f)$ the multiplicity of $g$ in the factorisation of $f$ in $\Q[\mathbf{t}]$. Also denote the polynomial in $\Z[\mathbf{t}]$, of the highest possible degree and smallest possible positive leading coefficient (in lexicographical ordering), that divides both $f$ and $g$ in $\Q[\mathbf{t}]$, by $\gcd(f,g)$.

From now on, we assume that $f,g\in\Z[t]$ are polynomials such that~\eqref{eq:commonrootfg} and~\eqref{eq:degfg} hold. Recall~\eqref{eq:defAB}.

\begin{lemma}\label{lemma:proppar}
Under the assumption~\eqref{eq:degfg}, if $\gcd(A^3,B^2)=S(x,y)$, then $\gcd(f^3,g^2)=S(t,1)$.
\end{lemma}
\begin{proof}
By the definition of~\eqref{eq:defAB}, the common factors of $A^3$ and $B^2$, other than possibly $y$, coincide with the common factor of $f^3$ and $g^2$.
Note that by~\eqref{eq:degfg},
 \[
 \min\left\{\frac{1}{2}\mult_y A,\frac{1}{3}\mult_y B\right\}=\varsigma-\tau\max\left\{\frac{1}{2}\deg f,\frac{1}{3}\deg g\right\}=\varsigma-\frac{2m}{\upsilon}=0,
\]
so $y$ is not a common factor of $A$ and $B$ in $\Q[x,y]$. 
\end{proof}

To understand the sets~\eqref{eq:tq} and~\eqref{eq:a1b1q}, we consider the region
\[
\cR_N\coloneqq\left\{(x,y)\in\R^2:
 |A(x,y)|\leq (N/4)^{\frac{1}{3}},\
 |B(x,y)|\leq (N/27)^{\frac{1}{2}}
\right\}.
\]
Define
\[
V\coloneqq\Vol(\cR_1).
\]

\begin{lemma}\label{lem:lenght}
 We have
 \begin{itemize}
 \item $\cR_N=\left\{\left(N^{\frac{\tau}{6\varsigma}}x,N^{\frac{1}{6\varsigma}}y\right):(x,y)\in\cR_1\right\}$;
 \item $\cR_1$ is bounded;
 \item $\Vol(\cR_N)=N^{\frac{\tau+1}{6\varsigma}}V$.
 \end{itemize} 
 In particular any $(x,y)\in\cR_N$ satisfies
 $x\ll N^{\frac{\tau}{6\varsigma}}$ and $y\ll N^{\frac{1}{6\varsigma}}$, where the implied constants depend only on $\cR_1$.
 \end{lemma}
 
\begin{proof}
Given $(x,y)\in \cR_N$, we can make a change of variables by taking $x=N^{\frac{\tau}{6\varsigma}}x'$ and $y=N^{\frac{1}{6\varsigma}}y'$. Then $A(x,y)=N^{\frac{1}{3}}y'^{2\varsigma}f(x'/y'^{\tau})$ and $B(x,y)=N^{\frac{1}{2}}y'^{3\varsigma}g(x'/y'^{\tau})$, so we precisely have $(x',y')\in \cR_1$.
By Lemma~\ref{lemma:proppar} and the assumption~\eqref{eq:commonrootfg}, $A(x,y)=B(x,y)=0$ has no solutions $(x,y)\in\R^2$ other than $(0,0)$.
Therefore, $\cR_1$ defines a bounded region.
By the change of variables, we see that
\[
\Vol(\cR_N)=\int_{(x,y)\in \cR_N}dxdy=N^{\frac{\tau+1}{6\varsigma}}\int_{(x',y')\in \cR_1}dx'dy'=N^{\frac{\tau+1}{6\varsigma}}V
\]
as required.
\end{proof}

The next lemma shows that omission of the $\Disc(E_{a,b})\neq 0$ condition from the definition of $\cA^{\delta}_{\upsilon}(N)$ introduces only a minor error.
\begin{lemma}\label{lemma:removingsingular}
 Let $P\in\Z[x,y]$ be a non-zero weighted homogeneous polynomial with weights $\tau,1$.
 Then
 \begin{gather*}
 \left|\{(a,b)\in \cT_{1,\tau}: P(a,b)=0\}\right|\ll 1,\\
 \left|\{(a,b)\in \Z^2_{\neq (0,0)}\cap \cR_N: P(a,b)=0\}\right|\ll N^{\frac{1}{6\varsigma}},
 \end{gather*}
 where the implied constant depends on $\cR_1$ and $P$.
\end{lemma}
\begin{proof}
 The polynomial $P$ can only have finitely many zeros $\alpha$ in $\PP^1_{\tau,1}(\Q)$. For each $\alpha$, fix a representative $(a,b)\in\Z^2_{\neq (0,0)}$ such that $p^{\tau}\nmid \gcd(a,b^{\tau})$. Then any other zeros of $P$ in the class $\alpha$ must be of the form $(\lambda^\tau a,\lambda b)$, with $\lambda\in\Z_{\neq 0}$. 
 For the first bound, $(a,b)\in \cT_{1,\tau}$ forces $\lambda\in\{\pm 1\}$, so the set is finite.
 For the second bound, Lemma~\ref{lem:lenght} implies that if $(\lambda^\tau a,\lambda b)\in\cR_N$, $\lambda^\tau a\ll N^{\frac{\tau}{6\varsigma}}$ and $\lambda b\ll N^{\frac{1}{6\varsigma}}$, so $\lambda\ll N^{\frac{1}{6\varsigma}}$.
\end{proof}

\subsection{Divisors of polynomial values}
We recall some standard properties of the resultant of polynomials. 

 \begin{lemma}\label{lemma:resultant}
 Let $F, G\in\Z[x,y]$ be homogeneous polynomials. Then 
 $\Res(F,G)=0$ if and only if $F$ and $G$ have a common zero in $\PP^1(\overline{\Q})$. 
In particular, any $a,b\in\Z$ satisfies
\[
\gcd(F(a,b),G(a,b))\mid \Res(F,G)\gcd(a,b)^{\deg F+\deg G-1}.
\]
\end{lemma}
\begin{proof}
 This follows from the properties of the resultant. See~\cite[Proposition~2.13]{SilvermanDynamical}.
\end{proof}

\begin{lemma}\label{lemma:resultantnonhom}
Let $P,Q\in\Z[x,y]$ be weighted homogeneous polynomials with weights $\tau,1$, and with no common zeros in $\PP_{(\tau,1)}^1(\Q)$. Let $i$ and $j$ be the weighted degrees of $P$ and $Q$, respectively.
Then there exists a positive integer $\Lambda$, depending on $P$ and $Q$, such that
\[
\gcd\left(P(a,b)^{j},Q(a,b)^{i}\right)^{\tau}\mid \Lambda\cdot \gcd\left(a,b^{\tau}\right)^{ij}
\]
holds
for any $(a,b)\in\Z^2_{\neq (0,0)}$.
\end{lemma}

\begin{proof} Let $F,G\in\Z[x,y]$ be homogeneous polynomials such that $P(x,y)^{\tau j}=F(x,y^{\tau})$ and $Q(x,y)^{\tau i}=G(x,y^{\tau})$, so $\deg F=\deg G= ij$. This is possible because the weighted degree of $P^{\tau j}$ and $Q^{\tau i}$ are both divisible by $\tau$.
Moreover $F,G\in\Z[x,y]$ have no common zeros in $\PP^1(\Q)$.
Let $a'=a/\gcd(a,b^{\tau})$ and $b'=b^{\tau}/\gcd(a,b^{\tau})$, so $\gcd(a',b')=1$.
Then
\begin{align*}
 P(a,b)^{\tau j}&=F(a,b^{\tau})=\gcd(a,b^{\tau})^{\deg F}F(a',b'),
\\
Q(a,b)^{\tau i}&=G(a,b^{\tau})=\gcd(a,b^{\tau})^{\deg G}G(a',b').
\end{align*}
Since $\gcd(a',b')=1$, Lemma~\ref{lemma:resultant} implies that
\[
\gcd(F(a',b'),G(a',b'))\mid \Res(F,G).
\]
Multiplying both sides by $\gcd(a,b^{\tau})^{ij}$ completes the proof.
\end{proof}

\begin{lemma}\label{lemma:resultantFG}
 Suppose that $f,g\in\Z[t]$ are coprime and satisfies~\eqref{eq:degfg}. 
 There exists a positive integer $\Lambda$ such that $n\mid \Lambda$ whenever $n^{12}\mid \gcd(A(a,b)^3,B(a,b)^2)$ for some $(a,b)\in\cT_{\upsilon,\tau}$.
\end{lemma}

\begin{proof}
Applying Lemma~\ref{lemma:resultantnonhom} to $A$ and $B$ with $i=2\varsigma$ and $j=3\varsigma$ yields a positive integer $\Lambda$ such that
\[
\gcd(A(a,b)^{3\varsigma},B(a,b)^{2\varsigma})^{\tau}\mid \Lambda\cdot \gcd(a,b^{\tau})^{6\varsigma^2}.
\]
By~\eqref{eq:degfg}, $\upsilon=2$ implies that $m=1$, so the definition of $\varsigma$ in~\eqref{eq:defAB} implies that
\[\frac{6\varsigma}{\tau}=\frac{12m}{\upsilon\tau}= \max\left\{3\deg f,2\deg g\right\}\in\Z.\] 
Therefore, we can take $\tau\varsigma$-th roots on both sides while keeping the term $\Lambda$ to get
\[
 \gcd(A(a,b)^{3},B(a,b)^{2})\mid \Lambda\cdot \gcd(a,b^{\tau})^{\frac{6\varsigma}{\tau}}. 
 \]

Suppose $p^{12k}\mid \gcd(A(a,b)^3,B(a,b)^2)$ for some prime $p$ and positive integer $k$ such that $k> v_p(\Lambda)$.
Then we must have
\[
p^{12}\mid\gcd(a,b^{\tau})^{\frac{6\varsigma}{\tau}}.
\]
Rearranging and putting in $\varsigma=2m/\upsilon$, we have
\[
p^{\left\lceil\frac{\upsilon\tau}{m}\right\rceil}\mid\gcd(a,b^{\tau}).
\]
The assumption~\eqref{eq:degfg} states that $\upsilon\tau=1$ or $m=1$, so $\left\lceil\frac{\upsilon\tau}{m}\right\rceil=\upsilon\tau$. This contradicts with $(a,b)\in\cT_{\upsilon,\tau}$. 
\end{proof}

We recall the following standard result. See for example~\cite[Lemma~4.5]{PSW}.

\begin{lemma}\label{lemma:rootbound}
Let $F\in\Z[x,y]$ be a homogeneous polynomial. Assume that $\Disc(F)\neq 0$.
There exists a constant $C$ depending only on the discriminant of $F$ such that for all $k\geq 1$,
\[
\left|\left\{(a,b)\bmod p^{k}:\gcd(a,b,p)=1,\ F(a,b)\equiv 0\bmod p^{k}\right\}\right|\leq Cp^k.
\]
\end{lemma}

\section{The Tamagawa ratio between isogenous elliptic curves}\label{sec:tamagawa}
The goal of this section is to give a description of the Tamagawa ratio of elliptic curves $\cE_t$ admitting a rational prime degree $\ell$ isogeny $\phi:\cE_t\rightarrow \cE_t'$. In particular, the Tamagawa ratio coincides with the ratio of the size of the $\ell$-isogeny Selmer groups of $\cE'_t$ and $\cE_t$, up to a bounded constant.

\subsection{Relating the $\ell$-isogeny Selmer group to the Tamagawa ratio}
Given an elliptic curve $E$ defined over $\Q$ and a prime $p$, we define the local Tamagawa number $c_p(E)$ to be the cardinality of $E(\Q_p)/E_0(\Q_p)$ (see~\cite[Section~VII.6]{arithmetic}).

\begin{lemma}\label{lemma:cpsel}
Let $\phi:E\rightarrow E'$ be a degree $\ell$ rational isogeny and let $\hat{\phi}:E'\rightarrow E$ be its dual. Then
\[
C\prod_{p\neq \ell} \frac{c_p(E')}{ c_p(E)}\leq \frac{|\Sel_{\phi}(E/\Q)|}{|\Sel_{\hat{\phi}}(E'/\Q)|}\leq C'\prod_{p\neq \ell} \frac{c_p(E')}{ c_p(E)},
\]
where $C,C'>0$ are constants depending only on $\ell$.
\end{lemma}

\begin{proof}
By Theorem~1.1, (1.22) and (3.4) in the work of Cassels~\cite{Cassels8}, we have
\[
\frac{|\Sel_{\phi}(E/\Q)|}{|\Sel_{\hat{\phi}}(E'/\Q)|}\cdot \frac{|E'(\Q)[\hat{\phi}]|}{|E(\Q)[\phi]| }=\prod_v \frac{|E'(\Q_v)/\phi(E(\Q_v))|}{|E(\Q_v)[\phi]|}.
\]
Notice $1\leq |E(\Q)[\phi]|,|E'(\Q)[\hat{\phi}]|\leq \ell$, so $|E'(\Q)[\hat{\phi}]|/|E(\Q)[\phi]|$ is bounded in terms of $\ell$.
By Lemma 4.2 and Lemma 4.3 in~\cite{DokchitserDokchitser}, when $v$ is the $p$-adic valuation with $p\neq \ell$, the local factor in the product on the right-hand side can be rewritten as
\[
\frac{\abs{E'(\Q_v)/\phi(E(\Q_v))}}{\abs{E(\Q_v)[\phi]}}=\frac{ c_p(E')}{c_p(E)}.
\]
We conclude by noting that the contribution from $v=\infty,\ell$ is bounded only in terms of $\ell$, since both $|E'(\Q_v)/\phi(E(\Q_v))|$ and $|E(\Q_v)[\phi]|$ are bounded by $\ell$.
\end{proof}

 \subsection{Factorisation of the discriminant polynomial}
Let $\ell$ be a prime. Let $f,g,f',g'\in\Z[t]$ such that 
\[
\cE:y^2=x^3+f(t)x+g(t)\qquad\text{ and }\qquad\cE':y^2=x^3+f'(t)x+g'(t)
\]
define elliptic curves over $\Q(t)$.
Assume that there exists a degree $\ell$ isogeny over $\Q(t)$ from $\phi:\cE\rightarrow \cE'$. 
Specialising to any $t$ such that $\Disc(\cE_t),\Disc(\cE'_t)\neq 0$, $\phi$ gives rise to a degree $\ell$ isogeny $\cE_t\rightarrow \cE_{t}'$.

Recall $A(t), B(t)\in\Z[t]$ from~\eqref{eq:defAB}, and similarly define
\[
A'(x,y)\coloneqq y^{2\varsigma'}f'(x/y^{\tau}) \qquad\text{ and }\qquad B'(x,y)\coloneqq y^{3\varsigma'}g'(x/y^{\tau}),\] where $\varsigma'$ is taken such that $\varsigma'\geq \tau\max\left\{\frac{1}{2}\deg f',\frac{1}{3}\deg g'\right\}$ and $\varsigma'\equiv \varsigma\bmod 2$.
Writing $t=a/b^{\tau}$ with $(a,b)\in\cT_{\upsilon,\tau}$, we consider
\[
E_{a,b}:y^2=x^3+A(a,b)x+B(a,b)\qquad\text{ and }\qquad E'_{a,b}:y^2=x^3+A'(a,b)x+B'(a,b).
\]
Lemma~\ref{lemma:maincorr} implies that $\cE_t^{b^{\varsigma}}\cong E_{a,b}$ and $(\cE'_t)^{b^{\varsigma}}\cong E'_{a,b}$ over $\Q$, so the degree $\ell$ isogeny $\cE_t\rightarrow \cE_{t}'$ translates to a rational degree $\ell$ isogeny $E_{a,b}\rightarrow E_{a,b}'$.

We will make use of~\cite[Theorem~6.1]{DokchitserDokchitser} to compute Tamagawa ratios. 
Let
\[
 \Delta(x,y)\coloneqq 4A(x,y)^3+27B(x,y)^2\quad\text{ and }\quad
 \Delta'(x,y)\coloneqq 4A'(x,y)^3+27B'(x,y)^2.
\]
\begin{lemma}\label{lemma:discrfactors}
 Let $P\in\Z[x,y]$ be any irreducible weighted homogeneous polynomial with weights $\tau,1$. Then one of the following holds:
 \begin{itemize}
 \item $P^4\mid A$ and $P^6\mid B$; 
 \item $P^4\mid A'$ and $P^6\mid B'$; 
 \item $\mult_P(\Delta)= \mult_P(\Delta')=0$;
 \item $P\mid A,B,A',B'$;
 \item $\mult_P(\Delta)/\mult_P(\Delta')=\ell^{\pm 1}$.
 \end{itemize}
\end{lemma}

\begin{proof}
 We claim that there exists $p\nmid 6\ell$ and $(a,b)\in \Z^2$ such that $v_p(P(a,b))=1$,
and $p\nmid Q(a,b)$ for any irreducible factor $Q\neq P$ of the product $ABA'B'\Delta\Delta'$ over $\Q[x,y]$.
 Since $P$ is weighted homogeneous and irreducible, at least one of $P(t,1)$ and $P(1,t)$ is non-constant. If $P(t,1)$ is non-constant, fix $b=1$, otherwise fix $a=1$. Without loss of generality, assume that $P(t,1)$ is non-constant; otherwise, we swap the roles of $a$ and $b$. It follows from Schur's Theorem~\cite{Schur} that there are infinitely many primes $p$ such that $p\mid P(\alpha,1)$ for some integer $\alpha$. Take $p$ to be such a prime, with the further condition that it does not divide the discriminant or the leading coefficient of the product of distinct irreducible factors of $ABA'B'\Delta\Delta'P(t,1)$. By Hensel's Lemma, there exists a simple $\ZZ_p$-root $\alpha$ to $P(t,1)$. Take $a$ to be an integer such that $a\equiv \alpha\bmod p$ and $a\not\equiv \alpha\bmod p^2$. Then $v_p(P(a,1))=1$, so $p$ and $(a,b)=(a,1)$ satisfy the required properties. This proves the claim and we henceforth fix $p$ and $(a,b)$.
 
By the construction of $(a,b)$, we have $v_p(P(a,b))=1$ and $p\neq 2,3$, so we deduce that $v_p(A(a,b))=\mult_P(A)$, $v_p(B(a,b))=\mult_P(B)$, $v_p(\Delta(a,b))=\mult_P(\Delta)$, and $v_p(\Delta'(a,b))=\mult_P(\Delta')$.
Assume that we are not in the first two cases, so $P^{12}$ does not divide both $A^3$ and $B^2$, also $P^{12}$ does not divide both $A'^3$ and $B'^2$. 
Therefore $p^{12}\nmid \gcd(A(a,b)^3,B(a,b)^2)$, so $E_{a,b}$ is a minimal model at $p$. Similarly, $E'_{a,b}$ is a minimal model at $p$. 
 
 By~\cite[Table~1]{DokchitserDokchitser}, $E_{a,b}$ has good (resp.~additive or multiplicative) reduction modulo $p$ if and only if $E_{a,b}'$ has good (resp.~additive or multiplicative) reduction modulo $p$.
 If both $E_{a,b}$ and $E_{a,b}'$ have good reduction at $p$, then $v_p(\Delta(a,b))=v_p(\Delta'(a,b))=0$ and hence $\mult_P(\Delta)=\mult_P(\Delta')=0$.
 Otherwise, $E_{a,b}$ and $E_{a,b}'$ both have bad reduction. In this case, we have
 $v_p(\Delta(a,b)), v_p(\Delta'(a,b))>0$ and hence $\mult_P(\Delta), \mult_P(\Delta')>0$.

 Observe that $E_{a,b}$ has additive reduction at $p\neq 3$ if and only if $A(a,b)\equiv B(a,b)\equiv 0\bmod p$ by~\cite[Proposition~VII.5.1]{arithmetic}. This translates to $P\mid A,B$. In this case $E_{a,b}'$ would also have additive reduction at $p$, so similarly $P\mid A',B'$. 
 
 Finally $E_{a,b}$ and $E_{a,b}'$ have multiplicative reduction when $p\mid \Delta(a,b)$ and $p\nmid A(a,b)$. Moreover
 \[
 \frac{\mult_P(\Delta)}{\mult_{P}(\Delta')}=\frac{v_{p}\left(\Delta(a,b)\right)}{v_{p}\left(\Delta'(a,b)\right)}
 \]
 and we conclude by~\cite[Theorem~5.1]{DokchitserDokchitser}.
\end{proof}
Assume that there exists no $P\in\Q[x,y]$ such that $P^{12}$ divides both $A^3$ and $B^2$, or both $A'^3$ and $B'^2$.
Collecting all the irreducible factors dividing the discriminant of $E_{a,b}$ and $E'_{a,b}$, Lemma~\ref{lemma:discrfactors} allows us to conclude that there exist $T,T',D_{\pm}\in\Z[x,y]$ and $c,c'\in\Z$ such that
\begin{equation}\label{eq:defDpm}
 \Delta(x,y)=c'T(x,y)D_+(x,y)D_-(x,y)^\ell,\qquad 
 \Delta'(x,y)=cT'(x,y)D_+(x,y)^\ell D_-(x,y),
\end{equation}
where \begin{itemize}
 \item $T$, $D_+$ and $D_-$ are pairwise coprime, and
 \item the irreducible factors of $\gcd(A,B)$, $\gcd(A',B')$, $T$, $T'$ coincide.
\end{itemize}

We will associate with $D_{\pm}$ some quantities.
\begin{definition}\label{def:uv}
Given any irreducible weighted homogeneous polynomial $R\in\Q[x,y]$ with weights $\tau,1$, let $L$ be the smallest field containing all zeros of $R$ in $\PP^1_{(\tau,1)}(\overline{\Q})$, and take $\alpha\in \PP^1_{(\tau,1)}(L)$ to be a zero of $R$, define when $\deg B$ is even, 
\[
\theta(R)\coloneqq\begin{cases}
 \hfil \frac{1}{2}&\text{if } 6B(\alpha)\notin L^2,\\ 
 \hfil 1&\text{if } 6B(\alpha)\in L^2.
\end{cases}
\]
 Extend $\theta$ to all non-zero weighted homogeneous polynomials $R\in\Q[x,y]$ by taking the sum of $\theta$ over all distinct irreducible factors of $R$ over $\Q$.
Define
$u(R)$ to be the number of distinct irreducible factors of $R$ in $\Q[x,y]$, and 
\[
 v(R)\coloneq \begin{cases}\hfil 
 \frac{1}{2}u(R)&\text{if }\deg B\text{ is odd},\\ 
 \hfil \theta(R)&\text{if }\deg B\text{ is even}.
\end{cases}
\]
Given~\eqref{eq:defDpm}, define
\[
u_{\pm}\coloneqq u(D_{\pm})\qquad\text{ and }\qquad v_{\pm}\coloneqq v(D_{\pm}).
\]
 For $i\in\{1,2\}$, let $D_{\pm}^{(i)}$ be the product of all irreducible polynomials dividing $D_{\pm}$ with multiplicity congruent to $i\bmod 2$. 
 Define
 \[
 u_{\pm}^{(1)}\coloneqq u\left(D_{\pm}^{(1)}\right)
 , \qquad u_{\pm}^{(2)}\coloneqq u\left(D_{\pm}^{(2)}\right)
 \quad \text{ and }\quad v_{\pm}^{(2)}\coloneqq v\left(D_{\pm}^{(2)}\right).
 \]
 Define
 \begin{equation}\label{eq:cpmdef}
 (c_+,c_-)\coloneqq
\begin{cases}
 \hfil \left(v_+,\ v_-\right)&\text{if }\ell\geq 3,\\
 \hfil \left(u_+^{(1)}+v_+^{(2)},\ u_-^{(1)}+v_-^{(2)}\right)&\text{if }\ell=2.
\end{cases} 
\end{equation}
\end{definition}

We wish to describe the local Tamagawa ratios in terms of values of the polynomials $D_+$, $D_-$, and $B$.
\begin{lemma}\label{lemma:Tate}
 Let $p>3$ be a prime.
 Let $E:y^2=x^3+Ax+B$ be an elliptic curve with $A,B\in\Z$ such that $p^{12}\nmid\gcd(A^3,B^2)$ holds. 
 Then $E$ has
 \begin{itemize}
 \item good reduction at $p$ if $p\nmid 4A^3+27B^2$;
 \item split multiplicative reduction at $p$ if $p\mid 4A^3+27B^2$ and $\leg{6B}{p}=1$;
 \item non-split multiplicative reduction at $p$ if $p\mid 4A^3+27B^2$ and $\leg{6B}{p}=-1$;
 \item additive reduction at $p$ if $p\mid 4A^3+27B^2$ and $p\mid A$.
 \end{itemize}
\end{lemma}
\begin{proof}
The condition $p^{12}\nmid\gcd(A^3,B^2)$ ensures that $y^2=x^3+Ax+B$ is locally minimal at $p$.
 The case distinction between good, multiplicative, and additive reduction is clear by~\cite[Proposition~VII.5.1]{arithmetic}. The splitting type in the case of multiplicative reduction follows from~\cite[Theorem~V.5.3(b)]{advanced}.
\end{proof}
Now, we describe when $E$ has multiplicative reduction in full generality, not only when $E$ is minimal as in the previous lemma.
\begin{lemma}\label{lem:multgen}
 Let $p>3$ be a prime.
 Let $E:y^2=x^3+Ax+B$ be an elliptic curve with $A,B\in\Z$. Then $E$ has multiplicative reduction at $p$ if and only if $12\mid 3v_p(A)=2v_p(B)<v_p(\Delta)$.
\end{lemma}
\begin{proof}
Let $m=\lfloor v_p(A)/4, v_p(B)/6\rfloor$. Take integers $A'$ and $B'$ such that $A=A'p^{4m}$ and $B=B'p^{6m}$. 
Then $E':y^2=x^3+A'x+B'$ is $\Q$-isomorphic to $E$, and hence has the same reduction type as $E$.
By Lemma~\ref{lemma:Tate}, $E'$ has multiplicative reduction at $p$ if and only if $p\mid 4A'^3+27B'^2$ and $p\nmid A',B'$. These conditions are equivalent to $v_p(A)=4m$, $v_p(B)=6m$, and $v_p(4A^3+27B^2)>12m$.
\end{proof}

\begin{lemma}\label{lemma:tamratiocond}
 Let $\ell$ be a prime. 
 Let $p$ be a prime such that $p\nmid 6\ell$. 
 Recall~\eqref{eq:defDpm}. Let $R(x,y)$ be the product of distinct irreducible factors of $D_+(x,y) D_-(x,y)$.
 Let $(a,b)\in\Z^2$. 
 Assume that $p\nmid\gcd(D_+(a,b),D_-(a,b))$ and $v_p(c)=v_p(c')=0$.
 If $3v_p(A(a,b))=2v_p(B(a,b))>0$, assume that either $3v_p(A(a,b))=v_p(\Delta(a,b))$, or $4\nmid v_p(A(a,b))$.
 If $\ell\in\{2,3\}$, assume further that $f$ and $g$ are coprime, and that $p\nmid \gcd(A(a,b),B(a,b))$.
 \begin{itemize}
 \item If $\ell\geq 3$, 
 \[
 \frac{c_p(E'_{a,b})}{c_p(E_{a,b})}=\begin{cases}
 \ell &\text{ if } D_+(a,b)\equiv 0\bmod p\text{ and }\leg{6B(a,b)}{p}=1,\\
 \ell^{-1} &\text{ if } D_-(a,b)\equiv 0\bmod p\text{ and }\leg{6B(a,b)}{p}=1,\\
 1 &\text{ otherwise}.
 \end{cases}
 \]
\item If $\ell=2$ and $p^2\nmid R(a,b)$, we have
 \[
 \frac{c_p(E'_{a,b})}{c_p(E_{a,b})}=\begin{cases}
 2 &\text{ if } p\mid D^{(1)}_+(a,b)\text{ or }\left(p\mid D^{(2)}_+(a,b)\text{ and }\leg{6B(a,b)}{p}=1\right),\\
 2^{-1} &\text{ if } p\mid D^{(1)}_-(a,b)\text{ or }\left(p\mid D^{(2)}_-(a,b)\text{ and }\leg{6B(a,b)}{p}=1\right),\\
 1 &\text{ otherwise}.
 \end{cases}
 \]
 \end{itemize}
\end{lemma}

\begin{proof}
For $\ell\in\{2,3\}$, it follows from the assumption that $p\nmid A(a,b)$ or $p\nmid B(a,b)$. By Lemma~\ref{lemma:Tate}, $E_{a,b}$ can only have good reduction or multiplicative reduction at $p$.

We use Lemma~\ref{lemma:Tate} and~\cite[Table~1]{DokchitserDokchitser}.
Since we have excluded the possibility of additive reduction when $\ell\in\{2,3\}$, the only case when $c_p(E_{a,b'})/c_p(E_{a,b})\neq 1$ for any $\ell$ is when $E_{a,b}$ has multiplicative reduction at $p$.
If $p\mid T(a,b)$, then $p$ divides both $A(a,b)$ and $B(a,b)$ In the light of Lemma~\ref{lem:multgen}, our assumption implies that $12\mid 3v_p(A(a,b))=2v_p(B(a,b))<v_p(\Delta(a,b))$ cannot hold, so $E_{a,b}$ cannot have multiplicative reduction when $p\mid T(a,b)$. Therefore for there to be multiplicative reduction, we must have $p\mid D_+(a,b)D_-(a,b)$. By assumption $p$ cannot divide both $D_+(a,b)$ and $D_-(a,b)$.
When $\ell\geq 3$, we conclude by combining Lemma~\ref{lemma:Tate} and~\cite[Table~1]{DokchitserDokchitser}.

Now suppose instead that $\ell=2$. 
 If $p\mid D_+(a,b)$ and $\leg{6B(a,b)}{p}=1$, $E_{a,b}$ has split multiplicative reduction, and we conclude as before. If $p\mid D_+(a,b)$ and $\leg{6B(a,b)}{p}=-1$, it has non-split multiplicative reduction and so $c_p(E_{a,b})/c_p(E_{a,b}')=2$ if $v_p(D_+(a,b))$ is odd and $c_p(E_{a,b})/c_p(E_{a,b}')=1$ otherwise. We conclude by noticing that $v_p(D_+^{(1)}(a,b))\equiv v_p(D_+(a,b))\mod 2$. The case $p\mid D_-(a,b)$ is analogous.
\end{proof}

\section{Lattice point counting}\label{sec:lattice}
The goal of this section is to demonstrate that $(a,b)\in\cA^{\delta}_1(N)$ are equidistributed across typical congruence classes modulo $q$.
 We divide the discussion into two parts: first we treat the case $m=\delta=1$ with $f$ and $g$ are coprime; then we handle the case $\tau=\upsilon=1$, where the coprimality condition on $f$ and $g$ is dropped. Throughout, we treat $f,g,m,\tau,\upsilon$ as fixed, and thus omit their dependence in the notation.
 
\subsection{The case $m=\delta=1$}
In this section we allow $\tau$ to be any positive integer and $\upsilon\in\{1,2\}$. Moreover, we assume that $f$ and $g$ are coprime.
Lemma~\ref{lemma:resultantFG} ensures that we can take $\Lambda$ to be the smallest positive integer such that 
\begin{equation}\label{def:Eps}
n^{12}\mid \gcd(A(a,b)^3,B(a,b)^2)\text{ for some }(a,b)\in\cT_{\upsilon,\tau}\Rightarrow n\mid \Lambda.
\end{equation}
From~\eqref{eq:a1b1q}, we can write
\[
 \cA^1_{\upsilon}(N,(a_1,b_1),q)=\left\{(a,b)\in\cT_{\upsilon,\tau}: \begin{array}{l}
(a,b)\equiv (a_1,b_1)\bmod q,\\
 (a,b)\in \cR_{Ne^{12}},\ e=m(A(a,b)^3,B(a,b)^2)
\end{array} \right\}.
\]
We have a decomposition
 \begin{equation}\label{eq:engammanh}
 |\cA^1_{\upsilon}(N, (a_1,b_1),q)|=\sum_{\gamma}\mu(\gamma)\sum_{n\mid \Lambda}\sum_{\substack{e\mid n}}\mu(n/e)\left|\cM_{\gamma,n}\left(Ne^{12},(a_1,b_1),q\right)\right|,
\end{equation}
where
\[
\cM_{\gamma,n}(N,(a_1,b_1),q)\coloneqq
 \left\{(a,b)\in \cR_{N}\cap\Z^2_{\neq(0,0)}:\begin{array}{l} 
 (a,b)\equiv(a_1,b_1)\bmod q,\\
 \gamma^{\upsilon\tau}\mid \gcd(a,b^{\tau}),\\ n^{12}\mid \gcd\left(A(a,b)^3, B(a,b)^2\right)
\end{array}
\right\}.
\]
Define
\[
\varrho(M)\coloneqq \frac{\left|\left\{(a,b)\in (\Z/M\Z)^2:
M^2\mid \gcd(A(a,b)^3, B(a,b)^2)
\right\}\right|}{M^2}
.
\]

\begin{lemma}\label{lemma:singlelatticenh}
Let $a_1$, $a_2$, $q$, $n$, and $\gamma$ be positive integers. Suppose that $\gcd(q,n\gamma)=1$ and that $n\mid \Lambda$. Then
\[
 \left|\cM_{\gamma,n}(N,(a_1,b_1),q)\right|
=\frac{N^{\frac{\tau+1}{6\varsigma}}V}{\gamma^{\upsilon(1+\tau)}q^2}\cdot \varrho\left(\frac{n^{6}}{\gcd(n^{2},\gamma^{\varsigma\upsilon})^3}\right)+O\left(\frac{N^{\frac{\tau}{6\varsigma}}}{\gamma}+1\right),
\]
where the implied constant depends only on $f$ and $g$.
\end{lemma}

\begin{proof}
Write $a=\gamma^{\upsilon\tau} a'$ and $b=\gamma^{\upsilon}b'$.
Then the condition $n^{12}\mid \gcd(A(a,b)^3, B(a,b)^2)$ becomes $n^{12}\mid \gamma^{6\varsigma\upsilon}\gcd(A(a',b')^3, B(a',b')^2)$, hence
\[
M^6 \mid \gcd\left(A(a',b')^3, B(a',b')^2\right),
\text{ where }
M\coloneqq\frac{n^{2}}{\gcd\left(n^{2},\gamma^{\varsigma\upsilon}\right)}.
\]
Fix any $(a_1',b_1')$ such that 
\[
\left(a_1',b_1'\right)\equiv \left(a_1/\gamma^{\upsilon\tau},b_1/\gamma^{\upsilon}\right)\bmod q,\ A\left(a_1',b_1'\right)\equiv 0\bmod M^2\text{ and }B\left(a_1',b_1'\right)\equiv 0\bmod M^3.
\]
Then $(a',b')\equiv (a_1',b_1')\bmod qM^3$ defines a shifted lattice of determinant $q^2M^6$. Moreover, by Lemma~\ref{lem:lenght}, $(a',b')$ lies in the region with volume $\Vol(\cR_N)/\gamma^{\upsilon(1+\tau)}$ and side lengths bounded by $O(N^{\frac{\tau}{6\varsigma}}/\gamma)$.
The claim follows from a result of Davenport~\cite{DavenportLipschitz} and summing over all possible $(a'_1,b'_1)\bmod qM^3$. Note that $M$ is bounded since $\Lambda$ in~\eqref{def:Eps} is finite.
\end{proof}

\begin{lemma}\label{lemma:nonhompoint}
Let $a_1,b_1,q$ be integers such that $\gcd(a_1,b_1,q)=1$. 
Assume that $\gcd(q,\Lambda)=1$. Then
\begin{equation}\label{eq:nhpoint}
 \abs{\cA^1_{\upsilon}(N,(a_1,b_1),q)}=\frac{CN^{\frac{\tau+1}{6\varsigma}}V}{q^2}
 \prod_{p\nmid q} \left(1-\frac{1}{p^{\upsilon(1+\tau)}}\right)
 +O\left(N^{\frac{\tau}{6\varsigma}}\log N\right),
 \end{equation}
 where the implied constant depends only on $f$ and $g$, 
 and
 \[
 C\coloneqq\sum_{n\mid\Lambda}n^{\frac{2(\tau+1)}{\varsigma}}\varrho(n^6)\prod_{p\mid n}\left(1-\frac{1}{p^{\frac{2(\tau+1)}{\varsigma}}} \right)
\left(1-
 \frac{1}{p^{\upsilon(1+\tau)}}\right)^{-1}\left(1-
 \frac{\varrho\left(\frac{n^{6}}{\gcd(n^{2},p^{\varsigma\upsilon})^3}\right)}{p^{\upsilon(1+\tau)}\cdot \varrho(n^6)}\right).
 \]
\end{lemma}

\begin{proof}
We would like to evaluate~\eqref{eq:engammanh}.
Since $\gcd(a_1,b_1,q)=1$, the conditions $(a,b)\equiv(a_1,b_1)\bmod q$ and 
 $\gamma^{\upsilon\tau}\mid \gcd(a,b^{\tau})$ implies that $\gcd(\gamma,q)=1$. Moreover by Lemma~\ref{lem:lenght} and since $e\mid \Lambda$ is bounded, we have $a\ll N^{\frac{\tau}{6\varsigma}}$ and $b\ll N^{\frac{1}{6\varsigma}}$, so for $\cM_{\gamma,n}(Ne^{12},(a_1,b_1),q)$ to be non-empty, we must have 
 $\gamma\ll N^{\frac{1}{6\upsilon\varsigma}}$.
Applying Lemma~\ref{lemma:singlelatticenh}, and again noting that $e\mid n$ and $n\mid \Lambda$ are bounded, we have
\[
\left|\cM_{\gamma,n}(Ne^{12},(a_1,b_1),q)\right|=\frac{N^{\frac{\tau+1}{6\varsigma}}e^{\frac{2(\tau+1)}{\varsigma}}V}{\gamma^{\upsilon(1+\tau)}q^2}\varrho\left(\frac{n^{6}}{\gcd(n^{2},\gamma^{\varsigma\upsilon})^3}\right)
+O\left(\frac{N^{\frac{\tau}{6\varsigma}}}{\gamma}+1\right).
\]
Then putting this back to~\eqref{eq:engammanh}, we have 
\begin{align*}
 & |\cA^1_{\upsilon}(N, (a_1,b_1),q)|\\
 & =\frac{N^{\frac{\tau+1}{6\varsigma}}V}{q^2}\sum_{n\mid\Lambda}\sum_{\substack{\gamma\ll N^{\frac{1}{6\varsigma\upsilon}}\\ \gcd(\gamma,q)=1}}\mu(\gamma)\sum_{\substack{e\mid n}}\mu(n/e)
 \frac{e^{\frac{2(\tau+1)}{\varsigma}}}{\gamma^{\upsilon(1+\tau)}}\varrho\left(\frac{n^{6}}{\gcd(n^{2},\gamma^{\varsigma\upsilon})^3}\right)+O\left(N^{\frac{\tau}{6\varsigma}}\log N\right)\\&
 =\frac{N^{\frac{\tau+1}{6\varsigma}}V}{q^2}
 \sum_{n\mid\Lambda}n^{\frac{2(\tau+1)}{\varsigma}}\prod_{p\mid n}\left(1-\frac{1}{p^{\frac{2(\tau+1)}{\varsigma}}} \right)
 \sum_{\substack{\gamma\ll N^{\frac{1}{6\upsilon\varsigma}}\\ \gcd(\gamma,q)=1}}
 \frac{\mu(\gamma)}{\gamma^{\upsilon(1+\tau)}}\varrho\left(\frac{n^{6}}{\gcd(n^{2},\gamma^{\varsigma\upsilon})^3}\right)\\&+O\left(N^{\frac{\tau}{6\varsigma}}\log N\right).
 \end{align*}
 Next we want to extend the sum over $\gamma\ll N^{\frac{1}{6\upsilon\varsigma}}$ to all positive integers $\gamma$. This introduces an error of
 \[
 \ll N^{\frac{\tau+1}{6\varsigma}}
 \sum_{\gamma\gg N^{{\frac{1}{6\upsilon\varsigma}}}}
 \frac{1}{\gamma^{\upsilon(1+\tau)}}\ll 
 N^{\frac{\tau+1}{6\varsigma}-\frac{\upsilon(1+\tau)-1}{6\upsilon\varsigma}}\ll N^{\frac{1}{6\upsilon\varsigma}},
 \]
where we have used the fact that $\Lambda$ is finite.
 Since $\varrho$ is multiplicative, we can compute the sum
 \begin{multline*}
 \frac{1}{\varrho(n^6)}\sum_{\substack{\gamma\\ \gcd(\gamma,q)=1}}
 \frac{\mu(\gamma)}{\gamma^{\upsilon(1+\tau)}}\varrho\left(\frac{n^{6}}{\gcd(n^{2},\gamma^{\varsigma\upsilon})^3}\right)
 =\prod_{p\nmid q}\left(1-
 \frac{\varrho\left(\frac{n^{6}}{\gcd(n^{2},p^{\varsigma\upsilon})^3}\right)}{p^{\upsilon(1+\tau)}\cdot \varrho(n^6)}\right)\\ 
 =\prod_{p\nmid q}\left(1-
 \frac{1}{p^{\upsilon(1+\tau)}}\right)
 \prod_{p\mid n}\left(1-
 \frac{1}{p^{\upsilon(1+\tau)}}\right)^{-1}\left(1-
 \frac{\varrho\left(\frac{n^{6}}{\gcd(n^{2},p^{\varsigma\upsilon})^3}\right)}{p^{\upsilon(1+\tau)}\cdot \varrho(n^6)}\right).
 \end{multline*}
 Putting this back, we have~\eqref{eq:nhpoint} as required.
\end{proof}

\begin{proposition}\label{prop:nonhomprob}
Let $a_1,b_1,q$ be positive integers such that $\gcd(a_1,b_1,q)=\gcd(q,\Lambda)=1$. 
Then for all $N$ such that $|\cA^1_{\upsilon}(N)|>0$, we have
 \[
 \frac{|\cA^1_{\upsilon}(N,(a_1,b_1),q)|}{|\cA^1_{\upsilon}(N)|}=\prod_{p\mid q}\frac{1}{p^2}\left(1-\frac{1}{p^{\upsilon(1+\tau)}}\right)^{-1}+O\left(N^{-\frac{1}{6\varsigma}}\log N\right),
 \]
 where the implied constant depends only on $f$ and $g$.
\end{proposition}

\begin{proof}
By Lemma~\ref{lemma:nonhompoint}, we have
\[ \abs{\cA^1_{\upsilon}(N,(a_1,b_1),q)}=\frac{CN^{\frac{\tau+1}{6\varsigma}}V}{q^2}
 \prod_{p\nmid q} \left(1-\frac{1}{p^{\upsilon(1+\tau)}}\right)
 +O\left(N^{\frac{\tau}{6\varsigma}}\log N\right).
 \]
 Putting in $q=1$, we get the estimate
\[ \abs{\cA^1_{\upsilon}(N)}=CN^{\frac{\tau+1}{6\varsigma}}V
 \prod_{p} \left(1-\frac{1}{p^{\upsilon(1+\tau)}}\right)
 +O\left(N^{\frac{\tau}{6\varsigma}}\log N\right),
 \]
 which can be rewritten as
 \[ \abs{\cA^1_{\upsilon}(N)}=CN^{\frac{\tau+1}{6\varsigma}}V
 \left(1+O\left(N^{-\frac{1}{6\varsigma}}\log N\right)\right)\prod_{p} \left(1-\frac{1}{p^{\upsilon(1+\tau)}}\right).
 \]
Taking the ratio of $|\cA^1_{\upsilon}(N,(a_1,b_1),q)|$ and $|\cA^1_{\upsilon}(N)|$ completes the proof.
\end{proof}
\subsection{The case $\tau=\upsilon=1$}
In this section, we fix $f,g\in\Z[t]$ that satisfy~\eqref{eq:commonrootfg} and~\eqref{eq:degfg} with $\delta\in\{0,1\}$, $\tau=1$, $\upsilon=1$, so $A,B\in\Z[x,y]$ are homogeneous.
Lemma~\ref{lemma:proppar} shows that the common factors of $A^3$ and $B^2$ correspond precisely to the common factors of $f^3$ and $g^2$. By Lemma~\ref{lem:lenght}, $(a,b)\in \cR_N$ implies $\abs{a},\abs{b}\ll N^{\frac{1}{12m}}$.

The goal of this section is to estimate $|\cA^{\delta}_{1}(N)|$ and $|\cA^{\delta}_{1}(N, \alpha,q)|$, as defined in~\eqref{eq:a1b1} and~\eqref{eq:tq}.
Whenever we write $[a:b]\in \PP^1(\Z/M\Z)$, we implicitly take a representative with $a,b\in\Z$ and $\gcd(a,b,M)=1$.
Recall the fact that
\[
\abs{\PP^1(\Z/M\Z)}=M\prod_{p\mid M}\frac{p+1}{p}\quad\text{ and }\quad \prod_p\left(1-\frac{1}{p^2}\right)=\frac{6}{\pi^2}.
\]

Let $q$ be a positive integer. Given any $\alpha=[a_1:b_1]\in\PP^1(\Z/q\Z)$, to estimate the sizes of the sets defined in~\eqref{eq:tq},
\begin{align}
 \cA^0_{1}(N, \alpha,q)&=\left\{(a,b)\in\Z_{\neq(0,0)}^2: \begin{array}{l}
 (a,b)\in \cR_{N},
 \\\gcd(a,b)=1,\ (a,b)\in\Z\cdot (a_1,b_1)\bmod q
\end{array} \right\},\label{eq:S1d0} \\
 \cA^1_{1}(N, \alpha,q)&=\left\{(a,b)\in\Z_{\neq(0,0)}^2: \begin{array}{l}
 (a,b)\in \cR_{Ne^{12}},\ e=m(A(a,b)^3,B(a,b)^2),
 \\\gcd(a,b)=1,\ (a,b)\in\Z\cdot (a_1,b_1)\bmod q
\end{array} \right\},\label{eq:S1d1} 
\end{align}
define 
\[
\cM_{\gamma,n}(N,\alpha,q)\coloneqq
 \left\{(a,b)\in \Z^2_{\neq(0,0)}:\begin{array}{l} 
 (a,b)\in \cR_{N},\ (a,b)\in\Z\cdot (a_1,b_1)\bmod q,\\
 \gamma\mid \gcd(a,b),\ n^{12}\mid \gcd(A(a,b)^3, B(a,b)^2)
\end{array}
\right\}.
\]

The points in $\cM_{\gamma,n}(N,\alpha,q)$ lie in a union of lattices. We first apply Davenport's geometry-of-numbers result to an arbitrary lattice in our setting. 

\begin{lemma}\label{lemma:singlelattice}
Let $M$ and $\gamma$ be positive integers. Suppose that $\gamma$ is squarefree and $[a_1:b_1]\in \PP^1(\Z/M\Z)$. Then
\[
 \left|\left\{(a,b)\in \Z^2_{\neq(0,0)}:\begin{array}{l} 
 (a,b)\in \cR_N,\ \gamma\mid \gcd(a,b),\\
 (a,b)\in\Z\cdot (a_1,b_1)\bmod M
\end{array}
\right\}\right|
=\frac{N^{\frac{1}{6m}}V\gcd(\gamma,M)}{\gamma^2M}+O\left(\frac{N^{\frac{1}{12m}}}{\gamma}\right),
\]
where the implied constant depends only on $\cR_1$. If $\gamma\gg N^{\frac{1}{12m}}$, then the set is empty.
\end{lemma}

\begin{proof}
Since $\gamma$ is squarefree, we deduce that $\frac{\gamma}{\gcd(\gamma,M)}$ and $\frac{M}{\gcd(\gamma,M)}$ must be coprime.
Writing $a=\gamma a'$ and $b=\gamma b'$, the condition $(a,b)\equiv \Z\cdot (a_1,b_1)\bmod M$ becomes
$(a',b')\equiv \Z\cdot (a_1,b_1)\bmod \frac{M}{\gcd(\gamma,M)}$, which defines a lattice of determinant $\frac{M}{\gcd(\gamma,M)}$. Moreover, $(a',b')$ lies in the region $\cR_{N/\gamma^{12m}}$, which has volume $N^{\frac{1}{6m}}V/\gamma^2$ and side lengths bounded by $O(N^{\frac{1}{12m}}/\gamma)$ by Lemma~\ref{lem:lenght}.
Note that since $(a,b)\in \cR_N$ implies that $|a|,|b|\ll N^{\frac{1}{12m}}$, the set is empty when $\gamma\gg N^{\frac{1}{12m}}$.
When $\gamma\ll N^{\frac{1}{12m}}$, the claim follows from a result of Davenport~\cite{DavenportLipschitz}. 
\end{proof}

\subsubsection{When $\delta=0$}
\begin{proposition}\label{prop:easySequi}
For all $N$ such that $\abs{ \cA^0_{1}(N)}>0$ and for any $\alpha\in\PP^1(\Z/q\Z)$, we have
 \[
 \frac{\abs{ \cA^0_{1}(N,\alpha,q)}}{\abs{ \cA^0_{1}(N)}}=\frac{1}{|\PP^1(\Z/q\Z)|}+O\left(N^{-\frac{1}{12m}}\log N\right),
 \]
 where the implied constant depends only on $f$ and $g$.
\end{proposition}

\begin{proof}
We can decompose~\eqref{eq:S1d0} as
\[
 | \cA^0_{1}(N,\alpha,q)|=\sum_{\gamma}\mu(\gamma)\left|\cM_{\gamma,1}(N,\alpha,q)\right|.
\]
By Lemma~\ref{lemma:singlelattice}, we have
\[
\left|\cM_{\gamma,1}(N,\alpha,q)\right|=\frac{N^{\frac{1}{6m}}V\gcd(\gamma,q)}{\gamma^2q}+O\left(\frac{N^{\frac{1}{12m}}}{\gamma}\right)
\]
and $\left|\cM_{\gamma,1}(N,\alpha,q)\right|=0$ for $\gamma\gg N^{\frac{1}{12m}}$.
Then we can compute 
\[
 | \cA^0_{1}(N,\alpha,q)|=\sum_{\gamma\ll N^{\frac{1}{12m}}}\mu(\gamma)\frac{N^{\frac{1}{6m}}V\gcd(\gamma,q)}{\gamma^2q}+O\left(N^{\frac{1}{12m}}\log N\right).
\]
 Notice that we can extend the sum to all $\gamma$ with a cost of $O(N^{\frac{1}{12m}})$, so
 \begin{align*}
 | \cA^0_{1}(N,\alpha,q)|&=\sum_{\gamma}\mu(\gamma)\frac{N^{\frac{1}{6m}}V\gcd(\gamma,q)}{\gamma^2q}+O\left(N^{\frac{1}{12m}}\log N\right)\\
 &=\frac{N^{\frac{1}{6m}}}{q}V\prod_{p}\left(1-\frac{\gcd(p,q)}{p^2}\right)+O\left(N^{\frac{1}{12m}}\log N\right)\\
 &=\frac{N^{\frac{1}{6m}}V}{|\PP^1(\Z/q\Z)|}\prod_{p}\left(1-\frac{1}{p^2}\right)+O\left(N^{\frac{1}{12m}}\log N\right).
 \end{align*}
 By taking $q=1$, we also have 
 \[
 | \cA^0_{1}(N)|=N^{\frac{1}{6m}}V\prod_{p}\left(1-\frac{1}{p^2}\right)+O\left(N^{\frac{1}{12m}}\log N\right).
 \]
 The claim follows from taking the ratio between $| \cA^0_{1}(N,\alpha,q)|$ and $| \cA^0_{1}(N)|$.
\end{proof}

\subsubsection{When $\delta=1$}
In this case, we allow $f$ and $g$ to have a common factor and we count the curves ordered by naive height. We loosely follow the approach in~\cite{MolnarVoight} in counting elliptic curves with a $7$-isogeny to obtain the required level of distribution.
\begin{lemma}\label{lemma:latticegcd}
Let $p$ be a prime.
Let $A,B,K\in\Z[x,y]$ be homogeneous polynomials such that $K^r=\gcd(A^3,B^2)$ for some positive integer $r$.
Let $R=\Res(A^3/K^r, B^2/K^r)$.
Define 
\begin{equation}\label{eq:defnuk}
 \nu(k)\coloneqq\max\left\{\left\lceil\frac{12k}{r}\right\rceil, v_p(R)\right\}.
\end{equation}
Let
 \[
 \cU\coloneqq\left\{(a,b)\in\Z^2:
 p^{4k}\mid A(a,b),\ p^{6k}\mid B(a,b)
 \right\}.
 \]
 If $(a_1,b_1)\in\cU$ and $\gcd(a_1,b_1,p)=1$, then $\left(a_1+p^{\nu(k)}\Z,\ b_1+p^{\nu(k)}\Z\right)\subseteq \cU$. 
\end{lemma}

\begin{proof}
Write $A^3=K^rA'$ and $B^2=K^rB'$ so that $A'$ and $B'$ have no common divisor of positive degree in $\Q[x,y]$.
 Note that $(a,b)\in\cU$ if and only if
\begin{equation}\label{eq:ineqvalu}
 rv_p(K(a,b))+\\
\min\left\{ v_p(A'(a,b)),\ v_p(B'(a,b))\right\}\geq 12k.
\end{equation}
Suppose $(a,b)=(a_1+p^{\nu(k)}m,b_1+p^{\nu(k)}n)$ for some integers $m,n$.
If $v_p(K(a,b))\geq \nu(k)$, then
$ r v_p(K(a,b))\geq r\nu(k)\geq 12k$,
 hence~\eqref{eq:ineqvalu} is satisfied and $(a,b)\in\cU$.
Suppose instead $v_p(K(a,b))<\nu(k)$.
It is clear that $K(a,b)\equiv K(a_1,b_1)\bmod p^{\nu(k)}$, so 
\begin{equation}\label{eq:vpS}
 v_p(K(a,b))=v_p(K(a_1,b_1)).
\end{equation}
By Lemma~\ref{lemma:resultant}, since $\gcd(a_1,b_1,p)=1$,
\[
\min\left\{ v_p\left(A'(a_1,b_1)\right),\ v_p\left(B'(a_1,b_1)\right)\right\}\leq v_p(R)\leq \nu(k).
\]
Since $(a,b)\equiv (a_1,b_1)\bmod p^{\nu(k)}$, we have
\begin{equation}\label{eq:vpABp}
 \min\{ v_p(A'(a,b)),\ v_p(B'(a,b))\}\geq \min\{ v_p(A'(a_1,b_1)),\ v_p(B'(a_1,b_1))\}.
\end{equation}
Finally note that $(a_1,b_1)$ satisfies~\eqref{eq:ineqvalu}, so combining~\eqref{eq:vpS} and~\eqref{eq:vpABp} shows that $(a,b)$ also satisfies~\eqref{eq:ineqvalu}.
\end{proof}

\begin{lemma}\label{lemma:rhoest}
Let $r\in\{2,3,4,6,12\}$. Let $K\in\Z[x,y]$ be a homogeneous polynomial that has no square divisor in $\Q[x,y]$.
 Assume that $K^r=\gcd(A^3,B^2)$ for some positive integer $r$.
Define
\begin{equation}\label{eq:Psi}
 \psi(n)\coloneqq\prod_{p\mid n} p^{\nu(v_p(n))},
\end{equation}
where $\nu$ is defined in~\eqref{eq:defnuk}
Then for all positive integers $n$, we have 
\[
\eta(n)\coloneqq
\abs{\left\{[a:b]\in\PP^1(\Z/\psi(n)\Z): 
 n^4\mid A(a,b),\ n^6\mid B(a,b)
 \right\}}\ll C^{\omega(n)},
\]
and
\begin{equation}\label{eq:rhohom}
\rho(n)\coloneqq \frac{\eta(n)}{|\PP^1(\Z/\psi(n)\Z)|}
\ll \frac{C^{\omega(n)}}{n^{\frac{12}{r}}},
\end{equation}
where $C$ and the implied constants depend only on $f$ and $g$.
\end{lemma}

\begin{proof}
Since $K(x,y)$ is squarefree, so $\Disc(K)\neq 0$.
Let $R$ be as defined in Lemma~\ref{lemma:latticegcd}.
Let $k\coloneqq v_p(n)$ and $k_0\coloneqq \max\left\{\left\lceil\frac{12k-v_p(R)}{r}\right\rceil,0\right\}$.
If $a,b$ are integers such that $\gcd(a,b,p)=1$, then the condition $p^{4k}\mid A(a,b),\ p^{6k}\mid B(a,b)$ implies that $p^{k_0}\mid K(a,b)$.
By definition $\nu(k)\geq k_0$. By Lemma~\ref{lemma:rootbound}, we have
\begin{align*}
& \left|\left\{(a,b)\in\left(\Z/p^{\nu(k)}\Z\right)^2:\gcd(a,b,p)=1,\ 
 p^{k_0}\mid K(a,b)
 \right\}\right|\\
 & \leq p^{2(\nu(k)-k_0)}\left|\left\{(a,b)\in\left(\Z/p^{k_0}\Z\right)^2:\gcd(a,b,p)=1,\ 
 p^{k_0}\mid K(a,b)
 \right\}\right|\\
 & \leq C' p^{2(\nu(k)-k_0)}p^{k_0}=C' p^{2\nu(k)-k_0},\end{align*}
 where $C'$ is a constant depending only on $\disc(K)$.
Dividing by $\varphi(p^{\nu(k)})=p^{\nu(k)}(1-p^{-1})$, we see that 
\[
\eta\left(p^{k}\right)\leq C' \frac{p^{\nu(k)-k_0}}{1-p^{-1}}.
\]
Since $r\mid 12$ by assumption, $\nu(k)=\max\{\frac{12k}{r},v_p(R)\}$ and $\nu(k)-k_0\leq \max\{\frac{v_p(R)}{r},v_p(R)\}$. 
Therefore, we have the bound
 \[
 \eta\left(p^{k}\right)\leq 2C' p^{v_p(R)}.
 \]
Setting $C=2C'$ and putting together distinct prime powers,
\[
\rho(n)\cdot |\PP^1(\Z/\psi(n)\Z)|=\eta(n)\leq R C^{\omega(n)}.
\]
Noting that $\psi(n)\gg n^{\frac{12}{r}}$ yields the desired result.
\end{proof}

\begin{lemma}\label{lemma:boundpower}
 Assume that $\deg \gcd(A^3,B^2)\leq \frac{24m}{2+\kappa}$ for some $\kappa>0$. If $(a,b)\in\Z^2\cap \cR_{Ne^{12}}$ satisfies $e^{4}\mid A(a,b)$ and $e^{6}\mid B(a,b)$ for some positive integer $e$, we have
\[e\ll N^{\frac{1}{6\kappa}}.\]
\end{lemma}
\begin{proof}
Write $S\coloneqq\gcd(A^3,B^2)$.
 Using Lemma~\ref{lem:lenght}, we deduce that
 \[
 e^{12}\ll S(a,b)\ll (Ne^{12})^{\frac{\deg S}{12m}}\ll (Ne^{12})^{\frac{2}{2+\kappa}}.
 \]
 Rearranging gives $e\ll N^{\frac{1}{6\kappa}}$.
\end{proof}

\begin{lemma}\label{lemma:profconvrho}
Let $s=\gcd(f^3,g^2)$.
Assume that
\begin{itemize}
 \item $\deg s\leq \frac{24m}{2+\kappa}$ for some $\kappa>m$, and
 \item if $\deg s>0$, then $s(t)$ is an $r$-th power of a squarefree polynomial in $\Q[t]$, where $r\in\{2,3,4,6\}$. 
\end{itemize}
Let 
\begin{equation}\label{eq:deflocalfac}
 \lambda(p)\coloneqq \left(1+\left(1-\frac{1}{p^{\frac{2}{m}}}\right)\sum_{k\geq 1}p^{\frac{2k}{m}}\rho(p^{k})\right)^{-1},
\end{equation}
where $\rho$ is as given in~\eqref{eq:rhohom}.
Then 
\begin{itemize}
\item $\xi\coloneqq \frac{12}{r}-\frac{2}{m}\geq\kappa/m>1$,
 \item $n^{\frac{2}{m}}\rho(n)\ll C^{\omega(n)}n^{-\xi}$ for all integers $n$,
 \item $1<\lambda(p)^{-1}=1+O(p^{-\xi})$ for all primes $p$, and
 \item 
$\prod_p \lambda(p)^{-1}$ is finite and non-zero.
\end{itemize}

\end{lemma}
\begin{proof}
If $\deg s>0$, write $s(t)=k(t)^r$, where $k(t)$ is a squarefree polynomial in $\Q[t]$. The assumption~\eqref{eq:commonrootfg} implies that $\deg k\neq 1$, so $2r\leq\deg s\leq \frac{24m}{2+\kappa}$, which implies that $r\leq \frac{12m}{2+\kappa}$. If $\deg s=0$, we set $r=1$ for the purpose of this proof.
By the definition of $\xi$, we have $r=\frac{12m}{2+m\xi}$. We deduce from $r=\frac{12m}{2+m\xi}\leq\frac{12m}{2+\kappa}$ that $\xi\geq\kappa/m>1$.
The bound $\rho(n)\ll n^{-\frac{12}{r}}C^{\omega(n)}$ from Lemma~\ref{lemma:rhoest} implies that
$p^{\frac{2k}{m}}\rho(p^{k})\ll p^{-k\xi}$, so $\sum_{k\geq 1}p^{\frac{2k}{m}}\rho(p^{k})\ll p^{-\xi}$. 
Therefore $1<\lambda(p)^{-1}=1+O( p^{-\xi})$ for all prime $p$, and the Euler product $\prod_p \lambda(p)^{-1}$ converges to a non-zero constant since $\xi>1$.
\end{proof}

\begin{lemma}\label{lemma:maincounting}
There exists $\epsilon>0$ depending only on $f$ and $g$ such that the following holds.
Fix a positive integer $q$ and $\alpha=[a_1:b_1]\in\PP^1(\Z/q\Z)$. Let $s\coloneqq \gcd(f^3,g^2)$.
Assume that
\begin{itemize}
 \item $\gcd\left(A(a_1,b_1)^3, B(a_1,b_1)^2,q\right)=1$,
 \item $\deg s\leq \frac{24m}{2+\kappa}$ for some $\kappa>m$, and
 \item if $\deg s>0$, then $s(t)$ is an $r$-th power of a squarefree polynomial in $\Q[t]$, where $r\in\{2,3,4,6\}$. 
\end{itemize}
Then for all $N\geq 2$, we have
\[
 \left|\cA^1_{1}(N, \alpha,q)\right|
 =\frac{6}{\pi^2}\frac{N^{\frac{1}{6m}}V}{|\PP^1(\Z/q\Z)|}\prod_{p\nmid q}\lambda(p)^{-1}\\
+O\left(N^{\frac{1}{6m}-\epsilon}\right),
\]
where $\lambda$ is defined in~\eqref{eq:deflocalfac},
and the implied constant depends only on $f$ and $g$.
\end{lemma}

\begin{proof}
By Möbius inversion, we can express the count of~\eqref{eq:S1d1} as
\begin{equation}\label{eq:decompS2}
 |\cA^1_{1}(N, \alpha,q)|=\sum_{\gamma}\mu(\gamma)\sum_{n}\sum_{\substack{e\mid n}}\mu(n/e)\left|\cM_{\gamma,n}(N e^{12},\alpha,q)\right|.
\end{equation}

Our goal is to evaluate~\eqref{eq:decompS2}.
The assumption $\gcd(A(a_1,b_1)^{3}, B(a_1,b_1)^{2},q)=1$ implies that $\gcd(n,q)=1$ for $\cM_{\gamma,n}(Ne^{12},\alpha,q)$ to be non-empty.
 Recall that $\psi$ is defined in~\eqref{eq:Psi}. Lemma~\ref{lemma:latticegcd} tells us that the condition $n^{4}\mid A(a,b)$ and $n^{6}\mid B(a,b)$ defines a union of lattices modulo $\psi(n)$.
 We apply Lemma~\ref{lemma:singlelattice} to each lattice with $M=q\psi(n)$.
By Lemma~\ref{lemma:rhoest}, the number of such lattices is
 \[
 \left|\left\{[a:b]\in\PP^1(\Z/\psi(n)\Z):
 n^{4}\mid A(a,b),\ n^{6}\mid B(a,b)
 \right\}\right|=\rho(n)\psi(n)\prod_{p\mid n}\left(1+\frac{1}{p}\right)\ll C^{\omega(n)}.
 \]
Summing over all the lattices, and recalling that $\gamma$ is squarefree, we have
 \[
 \left|\cM_{\gamma,n}(Ne^{12},\alpha,q)\right|\\
 =\frac{N^{\frac{1}{6m}}e^{\frac{2}{m}}V\gcd(\gamma,qn)}{\gamma^2q}\rho(n)\prod_{p\mid n}\left(1+\frac{1}{p}\right)
 +O\left(C^{\omega(n)}\frac{N^{\frac{1}{12m}}e^{\frac{1}{m}}}{\gamma}\right),
 \]
by Lemma~\ref{lemma:singlelattice}. 
We put this back into~\eqref{eq:decompS2} and split the sum according to whether $n\leq Y$ or $n>Y$, where $Y\coloneqq N^{\frac{1}{12m(\xi+m^{-1})}}$. Write
\begin{equation}\label{eq:breakn}
 \left|\cA^1_{1}(N, \alpha,q)\right|=\Psi(Y)+\Phi(Y),
\end{equation}
where
\begin{align*}
 \Psi(Y)\coloneqq \sum_{\substack{n\leq Y\\ \gcd(n,q)=1}} \sum_{\gamma}\mu(\gamma)\sum_{\substack{e\mid n}}\mu(n/e)|\cM_{\gamma,n}(Ne^{12},\alpha,q)|,\\
 \Phi(Y)\coloneqq \sum_{\substack{n>Y\\ \gcd(n,q)=1}} \sum_{\gamma}\mu(\gamma)\sum_{\substack{e\mid n}}\mu(n/e)|\cM_{\gamma,n}(Ne^{12},\alpha,q)|.
\end{align*}
We first consider the contribution from the terms with $n\leq Y$.
Note that since $(a,b)\in \cR_{Ne^{12}}$ implies that $|a|,|b|\ll (Ne^{12})^{\frac{1}{12m}}$ by Lemma~\ref{lem:lenght}, the set $\cM_{\gamma,n}(Ne^{12},\alpha,q)$ is empty when $\gamma\gg (Ne^{12})^{\frac{1}{12m}}$, so we may assume that $\gamma\ll (Ne^{12})^{\frac{1}{12m}}\leq(Nn^{12})^{\frac{1}{12m}}$. After evaluating the sum over $e$, we have
\begin{multline}\label{eq:presmalln}
 \Psi(Y)
=\frac{N^{\frac{1}{6m}}V}{q}\sum_{\substack{n\leq Y\\ \gcd(n,q)=1}}n^{\frac{2}{m}}\rho(n)
\prod_{p\mid n}\left(1+\frac{1}{p}\right)\left(1-\frac{1}{p^{\frac{2}{m}}}\right)\sum_{\gamma\ll (Nn^{12})^{\frac{1}{12m}}}\frac{\mu(\gamma)\gcd(\gamma,qn)}{\gamma^2}\\
+O\left(N^{\frac{1}{12m}}Y^{1+\frac{1}{m}}(\log N)^{C'}\right),
\end{multline}
where $C'>0$ is some constant.
If we extend the sum over $n\leq Y$ and $\gamma\ll (Nn^{12})^{\frac{1}{12m}}$ to all $n$ and $\gamma$ in the main term of~\eqref{eq:presmalln},
we have 
\begin{align*}
 &\frac{N^{\frac{1}{6m}}V}{q}\sum_{\substack{n\\ \gcd(n,q)=1}}n^{\frac{2}{m}}\rho(n)\prod_{p\mid n}\left(1+\frac{1}{p}\right)\left(1-\frac{1}{p^{\frac{2}{m}}}\right)\sum_{\gamma}\frac{\mu(\gamma)\gcd(\gamma,qn)}{\gamma^2}\\
 &=\frac{N^{\frac{1}{6m}}V}{q}\prod_{p}\left(1-\frac{\gcd(p,q)}{p^2}\right)\sum_{\substack{n\\ \gcd(n,q)=1}}n^{\frac{2}{m}}\rho(n)\prod_{p\mid n}\left(1-\frac{1}{p^{\frac{2}{m}}}\right)\\
 &=\frac{N^{\frac{1}{6m}}V}{q}\prod_{p}\left(1-\frac{1}{p^2}\right)\prod_{p\mid q}\left(\frac{p}{p+1}\right)\sum_{\substack{n\\ \gcd(n,q)=1}}n^{\frac{2}{m}}\rho(n)\prod_{p\mid n}\left(1-\frac{1}{p^{\frac{2}{m}}}\right)\\
 &=\frac{6}{\pi^2}\frac{N^{\frac{1}{6m}}V}{|\PP^1(\Z/q\Z)|}\prod_{p\nmid q}\left(1+\left(1-\frac{1}{p^{\frac{2}{m}}}\right)\sum_{k\geq 1}p^{\frac{2k}{m}}\rho(p^{k})\right),
\end{align*}
where the final product is finite and non-zero due to Lemma~\ref{lemma:profconvrho}.

The error in extending the sum to all $n$ and $\gamma$ is bounded by
\begin{multline}\label{eq:EulerE}
\frac{N^{\frac{1}{6m}}V}{q}\mathop{\sum_n\sum_\gamma}
\limits_{\substack{n>Y\text{ or }\gamma\gg \smash{(Nn^{12})^{\frac{1}{12m}}}\\ \gcd(n,q)=1}}n^{\frac{2}{m}}\rho(n)
\prod_{p\mid n}\left(1+\frac{1}{p}\right)\left(1-\frac{1}{p^{\frac{2}{m}}}\right)\frac{\mu(\gamma)\gcd(\gamma,qn)}{\gamma^2}\\\ll 
\frac{N^{\frac{1}{6m}}}{q}\mathop{\sum_n\sum_\gamma}
\limits_{\substack{n>Y\text{ or }\gamma\gg \smash{(Nn^{12})^{\frac{1}{12m}}}\\ \gcd(n,q)=1}}
\frac{(2C)^{\omega(n)}}{n^{\xi}}\frac{\mu(\gamma)^2\gcd(\gamma,qn)}{\gamma^2},
\end{multline}
where we have used the bound $n^{\frac{2}{m}}\rho(n)\ll C^{\omega(n)}n^{-\xi}$ from Lemma~\ref{lemma:profconvrho} and the trivial bound $\prod_{p\mid n} \left(1+p^{-1}\right)\ll 2^{\omega(n)}$.
To bound the terms with $n>Y$ in~\eqref{eq:EulerE}, write $\nu=\gcd(\gamma,qn)$ and $\gamma=\gamma'\nu$, then we have
\[
\sum_{\gamma}\frac{\mu^2(\gamma)\gcd(\gamma,qn)}{\gamma^2}\ll\sum_{\nu\mid qn}\frac{\mu^2(\nu)}{\nu}\sum_{\gamma'}\frac{1}{\gamma'^2}\ll\prod_{p\mid qn}\left(1+\frac{1}{p}\right).
\]
Using $\prod_{p\mid n} \left(1+p^{-1}\right)\ll 2^{\omega(n)}$ and $\prod_{p\mid q} \left(1+p^{-1}\right)\ll q$, we can bound the terms with $n>Y$ in~\eqref{eq:EulerE} by
\[
\ll N^{\frac{1}{6m}}\sum_{n> Y}\frac{(4C)^{\omega(n)}}{n^{\xi}}
\ll \frac{N^{\frac{1}{6m}}(\log Y)^{4C}}{Y^{\xi-1}}.
\]
Similarly, the terms with $\gamma\gg (Nn^{12})^{\frac{1}{12m}}$ in~\eqref{eq:EulerE} is bounded by
\[
\ll 
\frac{N^{\frac{1}{6m}}}{q}\sum_n \frac{(2C)^{\omega(n)}}{n^{\xi}}\sum_{\nu\mid qn}\frac{\mu^2(\nu)}{\nu}\frac{1}{(Nn^{12}/\nu)^{\frac{1}{12m}}}\ll N^{\frac{1}{12m}}\sum_{n}\frac{(4C)^{\omega(n)}}{n^{\xi+\frac{1}{m}}}\ll N^{\frac{1}{12m}},
\]
where the last inequality follows since $\xi>1$.
Putting the estimates back into~\eqref{eq:presmalln}, we have
\begin{equation}
 \Psi(Y)=\frac{6}{\pi^2}\frac{N^{\frac{1}{6m}}V}{|\PP^1(\Z/q\Z)|}\prod_{p\nmid q}\lambda(p)^{-1}
+O\left(\left(\frac{N^{\frac{1}{6m}}}{Y^{\xi-1}}+N^{\frac{1}{12m}}Y^{1+\frac{1}{m}}\right)(\log N)^{C'}\right)\label{eq:sumsmalln}
\end{equation}
by enlarging $C'$ if necessary.

We now turn to the sum over $n>Y$.
We claim that
\begin{equation}\label{eq:sumlargen}
 \Phi(Y)\ll \frac{N^{\frac{1}{6m}}(\log Y)^{C'}}{Y^{\xi-1}}
+N^{\frac{1}{6\kappa}}(\log N)^{C'}
\end{equation}
for some constant $C'$.
 This essentially follows from~\cite[Lemma~4.2.14]{MolnarVoight}, but we give a slightly different proof here.
Dropping the condition $(a,b)\in\Z\cdot (a_1,b_1)\bmod q$, it suffices to bound
\[
 \sum_{n>Y}\sum_{\substack{e\mid n}}\mu(n/e)^2\left|\left\{(a,b)\in\Z^2\cap \cR_{Ne^{12}}:\begin{array}{l}
 \gcd(a,b)=1, \
 n^{4}\mid A(a,b),\ n^{6}\mid B(a,b)
\end{array}
\right\}\right|.
\]
By Lemma~\ref{lem:lenght}, we know that $\cR_{Ne^{12}}=(Ne^{12})^{\frac{1}{12m}}\cR_1$. 
Fix the smallest ellipse that covers $\cR_1$, then scaling this ellipse in all directions by $(Ne^{12})^{\frac{1}{12m}}$ covers $\cR_{Ne^{12}}$.
Applying~\cite[Lemma~2]{HBlattice} to $O(C^{\omega(n)})$-many determinant $\psi(n)$ lattices leads us to
\[
\left|\left\{(a,b)\in\Z^2\cap \cR_{Ne^{12}}:\begin{array}{l}
 \gcd(a,b)=1, \\
 n^{4}\mid A(a,b),\ n^{6}\mid B(a,b)
\end{array}
\right\}\right|\ll \left(\frac{N^{\frac{1}{6m}}e^{\frac{2}{m}}}{\psi(n)}+1\right)C^{\omega(n)}.
\]
By Lemma~\ref{lemma:boundpower} and the assumption $\deg s\leq \frac{24m}{2+\kappa}$, we have
 $n\ll N^{\frac{1}{6\kappa}}$.
Therefore we only have to sum the upper bound over $Y<n\ll N^{\frac{1}{6\kappa}}$, this gives
 \[
 \sum_{Y<n\ll N^{\frac{1}{6\kappa}}}\sum_{e\mid n}\mu(n/e)^2\left(\frac{N^{\frac{1}{6m}}e^{\frac{2}{m}}}{\psi(n)}+1\right)C^{\omega(n)}\ll N^{\frac{1}{6m}}
 \sum_{n>Y}\frac{(2C)^{\omega(n)}}{n^{2(\frac{6}{r}-\frac{1}{m})}}
 +\sum_{n\ll N^{\frac{1}{6\kappa}}}(2C)^{\omega(n)}
\]
recalling that $\psi(n)\gg n^{\frac{12}{r}}$.
We use Rankin's trick to bound the first sum. Take $\epsilon_1=\frac{1}{\log Y}$, we have
\[
\sum_{n>Y}\frac{(2C)^{\omega(n)}}{n^{\xi}}\leq \sum_{n}\frac{(2C)^{\omega(n)}}{n^{\xi}}\left(\frac{n}{Y}\right)^{\xi-1-\epsilon_1}
\ll \frac{1}{Y^{\xi-1}}\sum_{n}\frac{(2C)^{\omega(n)}}{n^{1+\epsilon_1}}\ll \frac{(\log Y)^{2C}}{Y^{\xi-1}},
\]
where the last inequality follows from the fact that $\zeta(s)$ has a simple pole at $s=1$.
The second sum is 
\[
\sum_{n\ll N^{\frac{1}{6\kappa}}}(2C)^{\omega(n)}\ll N^{\frac{1}{6\kappa}}(\log N)^{2C}.
\]
This proves~\eqref{eq:sumlargen}.

Finally we evaluate~\eqref{eq:breakn} using~\eqref{eq:sumsmalln} and~\eqref{eq:sumlargen} with the choice $Y\coloneqq N^{\frac{1}{12m(\xi+m^{-1})}}$. The main term matches, and the error term is of the form
\[
\left(N^{\frac{1}{12m}\left(1+\frac{1+m^{-1}}{\xi+m^{-1}}\right)}+N^{\frac{1}{6\kappa}}\right)(\log N)^{C}
\]
for some constant $C>0$. The exponents in $N$ of the error term are strictly less than $\frac{1}{6m}$ since $\xi>1$ and $\kappa>m$. Therefore it is possible to pick 
\[
0<\epsilon<\frac{1}{6m}-\max\left\{\frac{1}{12m}\left(1+\frac{1+m^{-1}}{\xi+m^{-1}}\right),\frac{1}{6\kappa}\right\}
\]
independently of $\alpha$ and $q$. This completes the proof.
\end{proof}

\begin{proposition}\label{prop:ratio}
Suppose we are in the setting of Lemma~\ref{lemma:maincounting}. There exists $\epsilon>0$, depending only on $f$ and $g$, such that whenever $\left|\cA^1_{1}(N)\right|>0$, we have
\[ 
\frac{\left|\cA^1_{1}(N, \alpha,q)\right|}{\left|\cA^1_{1}(N)\right|}
 =\frac{1}{|\PP^1(\Z/q\Z)|}\prod_{p\mid q}\lambda(p)+O(N^{-\epsilon}),
\]
where $\lambda$ is defined in~\eqref{eq:deflocalfac}
and the implied constant depends only on $f,g$.
\end{proposition}

\begin{proof}
By Lemma~\ref{lemma:maincounting}, we have
\[
 \left|\cA^1_{1}(N, \alpha,q)\right|
 =\frac{6}{\pi^2}\frac{N^{\frac{1}{6m}}V}{|\PP^1(\Z/q\Z)|}\prod_{p\nmid q}\lambda(p)^{-1}
+O\left(N^{\frac{1}{6m}-\epsilon}\right).
\]
Putting in $q=1$, we have
\[
 \left|\cA^1_{1}(N)\right|
 =\frac{6}{\pi^2}N^{\frac{1}{6m}}V\prod_{p}\lambda(p)^{-1}
+O\left(N^{\frac{1}{6m}-\epsilon}\right).
\]
Since $\lambda(p)^{-1}\geq 1$ for all $p$, we can rewrite this as
\[
 \left|\cA^1_{1}(N)\right|
 =\left(1+O\left(N^{-\epsilon}\right)\right)N^{\frac{1}{6m}}V\frac{6}{\pi^2}\prod_{p}\lambda(p)^{-1}.
\]
Since $\left|\cA^1_{1}(N)\right|\neq 0$, we can take the ratio of the above expressions of $\left|\cA^1_{1}(N, \alpha,q)\right|$ and $\left|\cA^1_{1}(N)\right|$. This yields
\[\frac{\left|\cA^1_{1}(N, \alpha,q)\right|}{\left|\cA^1_{1}(N)\right|}
=\frac{(1+O\left(N^{-\epsilon}\right))\prod_{p\mid q}\lambda(p)}{|\PP^1(\Z/q\Z)|
}+O\left(N^{-\epsilon}\right)=\frac{\prod_{p\mid q}\lambda(p)}{|\PP^1(\Z/q\Z)|
}+O\left(N^{-\epsilon}\right)\]
as required.
\end{proof}

\section{The distribution of the Tamagawa ratio}\label{sec:central}
\subsection{A central limit theorem}
We use the method of moments to prove a central limit theorem suited for our application. We largely follow the method of moments approach in proving Erd\H{o}s--Kac theorem as given in Chapter~2 of~\cite{Kowalskibook}. Also see~\cite{GS} for some general statements. 

\begin{theorem}\label{theorem:EKver}
Let $n$ be a positive integer and let $C,C',\kappa\geq 1$.
Let $\cA$ be an infinite set with a given height function. Let $\cA(N)$ be the set containing all elements of $\cA$ up to height $N$.
 Let $\rf_p:\cA\rightarrow \{c\in\ZZ: |c|\leq C\}$ be a function.
 For every prime $p$, let $\rY_p$ be a random variable on a probability space with probability measure $\mathbf{P}$. Assume that $(\rY_p)_p$ are mutually independent and $\abs{\rY_p}\leq C$.
 Let $E(N)$ and $Q\coloneqq Q(N)$ be positive real-valued functions such that $\log Q=o(\log E(N))$.
 Assume that for every squarefree positive integer $q$ and every $\omega(q)$-tuple $(\lambda_p)_{p\mid q}$ of non-zero integers, 
 \begin{equation}\label{eq:probindest}
 \frac{\left|\left\{x\in \cA(N):
(\rf_p(x))_{p\mid q}=(\lambda_p)_{p\mid q} \right\}\right|}{|\cA(N)|}=
 \prod_{p\mid q}\mathbf{P}\left(\rY_p=\lambda_p\right)+O\left(\frac{q^\kappa}{E(N)}\right)
 \end{equation}
 holds uniformly for all $N\geq C'$.
 Let 
 \[
 \mu(Q)\coloneqq\sum_{p\leq Q} \mathbf{E}[\rY_p]\quad \text{ and }\quad \sigma(Q)^2\coloneqq \sum_{p\leq Q}\mathbf{E}[(\rY_p-\mathbf{E}[\rY_p])^2].
 \]
 Assume that $\sigma(Q)\rightarrow\infty$ as $N\rightarrow\infty$.
 Equip $\cA(N)$ with the uniform probability measure.
 Then the random variable $\rX_N$ on $\cA(N)$ defined by
 \[
 x\mapsto\frac{\sum_{p\leq Q}\rf_p(x)-\mu(Q)}{\sigma(Q)}
 \]
 converges in distribution to a standard Gaussian random variable as $N\rightarrow \infty$.
\end{theorem}

\begin{proof}
Fix a non-negative integer $k$. 
Write $\mu_p\coloneqq\mathbf{E}[\rY_p]$. Let $\PP_N$ denote the uniform measure on $\cA(N)$ and $\EE_N$ denote the corresponding expectation.
 By the linearity of expectation, we have
 \begin{equation}\label{eq:firstexp}
 \EE_N\left[\left(\sum_{p\leq Q}\rf_p-\mu(Q)\right)^k\right]=\sum_{p_1\leq Q}\dots \sum_{p_k\leq Q}\EE_N[(\rf_{p_1}-\mu_{p_1})\dots(\rf_{p_k}-\mu_{p_k})]
 .
 \end{equation}
 The assumption~\eqref{eq:probindest} with $q=\lcm(p_1,\dots,p_k)$ implies that
\begin{align*}
 \EE_N\left[\prod_{i}^k\rf_{p_i}\right]&=
\sum_{\lambda\neq 0}\lambda\cdot \PP_N\left(\prod_{i}^k\rf_{p_i}=\lambda\right)
\\&=
\sum_{\prod_{p\mid q}\lambda_{p}\neq 0}\prod_{i=1}^k\lambda_{p_i} \prod_{p\mid q}\mathbf{P}\left(\rY_{p}=\lambda_p\right)+O\left(\frac{(2C)^{\omega(q)}q}{E(N)}\right).
\end{align*}
Then
 \[\EE_N\left[\prod_{i}^k\rf_{p_i}\right]=
\mathbf{E}\left[\prod_{i}^k\rY_{p_i}\right]+O\left(\frac{(2C)^{\omega(q)}q^{\kappa}}{E(N)}\right).\]
 Applying this with $q\mid\lcm(p_1,\dots,p_k)$ allows us to rewrite~\eqref{eq:firstexp} as
 \[ \EE_N\left[\left(\sum_{p\leq Q}\rf_p-\mu(Q)\right)^k\right]
 =\sum_{p_1\leq Q}\dots \sum_{p_k\leq Q}\mathbf{E}[(\rY_{p_1}-\mu_{p_1})\dots(\rY_{p_k}-\mu_{p_k})]+O\left(\frac{(4CQ)^{k+\kappa}}{E(N)}\right)
 .
 \]
 Again by linearity and the assumption $\log Q=o(\log E(N))$,
 \[
 =\mathbf{E}\left[\left(\sum_{p\leq Q}(\rY_{p}-\mu_{p})\right)^k\right]+O\left(\frac{(4CQ)^{k+1}}{E(N)}\right)
 =\mathbf{E}\left[\left(\sum_{p\leq Q}(\rY_{p}-\mu_{p})\right)^k\right]+O_k(1)
 .
 \]
Apply the central limit theorem~\cite[Theorem~B.7.2]{Kowalskibook}
to the random variables $\rY_p$, so $\rT_N\coloneqq (\sum_{p\leq Q} \rY_p-\mu(Q))/\sigma(Q)$ converges in distribution to a standard Gaussian random variable as $N\rightarrow \infty$.
Then the converse of method of moments~\cite[Theorem~B.5.6]{Kowalskibook}
 implies that the moments of $T_N$ tend to the moments of a standard Gaussian random variable.
Since $\EE_N[\rX_N^k]=\mathbf{E}[\rT_N^k]+O_k(\sigma(Q)^{-k})$ and $\sigma(Q)\rightarrow\infty$, $\EE_N[\rX_N^k]$ also converges to the moments of a standard Gaussian random variable.
The method of moments~\cite[Theorem~B.5.5]{Kowalskibook} implies that $\rX_N$ converges in distribution to the standard Gaussian random variable.
\end{proof}
\subsection{Bounding the average Tamagawa ratio}
\begin{theorem}\label{theorem:NT}
Let $\alpha> 1$.
 Let $\cA$ be an infinite set with a given height function. Let $\cA(N)$ contain all elements of $\cA$ up to height $N$ and assume that $\log |\cA(N)|\asymp \log N$.
 For each $p$, take a function $\rf_p:\cA\rightarrow \{-1,0,1\}$. Suppose that there exists $\xi>0$ and multiplicative functions $\eta_+$ and $\eta_-$ taking values in $[0,1]$ such that for every pair of coprime positive squarefree integers $q_-,q_+$ with $q_-q_+\leq N^{\xi/2}$, we have
 \begin{equation}\label{eq:leveldist}
 \frac{1}{|\cA(N)|}\left|\left\{x\in \cA(N):
\begin{array}{l}
 p\mid q_-\Rightarrow \rf_p(x)=-1\\
 p\mid q_+\Rightarrow \rf_p(x)=1\\
\end{array}\right\}\right|=
 \eta_-(q_-) \eta_+(q_+)+O\left(\frac{1}{N^{\xi}}\right).
 \end{equation}
Assume further that there exist constants $c_+,c_-,C >0$, and $M_-,M_+$ such that
\begin{gather}\label{eq:assumeNT0}
\eta_{\pm}(p)\ll \frac{1}{p},\\\label{eq:assumeNT1}
\sum_{p\leq N}\eta_\pm(p)=c_{\pm}\log\log N+M_{\pm}+O\left(\frac{1}{\log N}\right),\\
\max_{x\in\cA(N)}\prod_{\rf_p(x)\neq 0}p\ll N^{C} \label{eq:assumeNT2}
\end{gather}
hold uniformly for all $N\geq 2$.
 Then
\[|\cA(N)|(\log N)^{(\alpha -1)c_+-(1-\alpha^{-1})c_-}\ll \sum_{x\in\cA(N)}\prod_p\alpha^{\rf_p(x)}
\ll |\cA(N)|(\log N)^{(\alpha -1)c_+},\]
where the implied constants depends only on $\xi, c_\pm, M_\pm, C, \alpha$ and the implied constants in the assumptions.
\end{theorem}
\begin{proof}
For any $x\in\cA(N)$, write 
 \[q_+(x)\coloneqq\prod_{\rf_p(x)=1}p\quad\text{ and }\quad q_-(x)\coloneqq\prod_{\rf_p(x)=-1}p.\]
As for the upper bound, notice that $\sum_p\rf_p(x)=\omega(q_+(x))-\omega(q_-(x))\leq \omega(q_+(x))$, so trivially we have
\[\sum_{x\in\cA(N)}\prod_p\alpha^{\rf_p(x)}
\leq
\sum_{x\in\cA(N)}\alpha^{\omega(q_+(x))}.
\]
Our assumptions allow us to apply~\cite[Theorem~1.9]{MR4870247} with $\chi_N(x)=\mathbf{1}_{x\in\cA(N)}$, $c_x=q_+(x)$,

$f(n)=\alpha^{\omega(n)}$, $M(N)=|\cA(N)|$, and $h_N(q_+)=\eta_+(q_+)$ to get
\[\sum_{x\in\cA(N)}\prod_p\alpha^{\rf_p(x)}
\leq \sum_{x\in\cA(N)}\alpha^{\omega(q_+(x))}
\ll |\cA(N)|(\log N)^{\alpha c_+-c_+},\]
which gives the required upper bound.

 To prove the lower bound, we carry out a modification of~\cite[Section 4]{MR4870247}.
 Given any integer $n$, denote the largest prime factor of $n$ by $P^+(n)$, and denote the smallest prime factor of $n$ by $P^-(n)$.
 We will split up the sum over $x\in\cA(N)$ in terms of the smooth part and the rough part of $q_-(x)$ and $q_+(x)$.
 Consider
 \begin{equation}\label{eq:NT1}
 \sum_{x\in\cA(N)}\prod_p\alpha^{\rf_p(x)}
 =\sum_{\substack{d_+\\P^+(d_+)\leq z}}\sum_{\substack{d_-\\P^+(d_-)\leq z}}\mu(d_+d_-)^2\sum_{\substack{x\in\cA(N)\\d_{\pm}\mid q_{\pm}(x)\\ P^-(q_{\pm}(x)/d_\pm)>z}}\prod_p\alpha^{\rf_p(x)},
 \end{equation}
 where $z\coloneqq N^{\nu}$ and $0<\nu<\xi$. By~\eqref{eq:assumeNT2}, the number of $p>z$ such that $f_p(x)\neq 0$ is at most $C/\nu +O(1)$, so $\sum_{p>z}|\rf_p(x)|\ll 1$.
 This leads to the lower bound
 \[\prod_p\alpha^{\rf_p(x)}\gg \prod_{p\leq z}\alpha^{\rf_p(x)}.\]
 Since we are aiming for a lower bound to~\eqref{eq:NT1}, we are allowed to restrict the sum to $d_+d_-\leq N^{\epsilon}$ for some $\epsilon>0$, so we have
\begin{equation}\label{eq:NT2}
 \sum_{x\in\cA(N)}\prod_p\alpha^{\rf_p(x)} \gg \sum_{\substack{(d_+,d_-)\\ d_+d_-\leq N^{\epsilon}\\P^+(d_+d_-)\leq z}}\mu(d_+d_-)^2\alpha^{\omega(d_+)-\omega(d_-)}\sum_{\substack{x\in\cA(N)\\d_{\pm}\mid q_{\pm}(x)\\ P^-(q_{\pm}(x)/d_\pm)>z}}1.
\end{equation}

To lower bound the inner sum, we invoke the fundamental lemma of sieve theory. Take the sequence $\lambda_m^-$ from~\cite[Lemma 6.3]{IKbook} with $y\coloneqq N^{\theta}$ and $\nu<\theta<\xi/2$. The property~\cite[(6.46)]{IKbook} allows us to write
\begin{equation}\label{eq:sieveinq}
 \sum_{\substack{x\in\cA(N)\\d_{\pm}\mid q_{\pm}(x)\\ P^-(q_{\pm}(x)/d_\pm)>z}}1
\gg \sum_{\substack{m\\P^+(m)\leq z}}\mu(m)^2\lambda_m^-\sum_{\substack{x\in\cA(N)\\d_{\pm}\mid q_{\pm}(x)\\md_-d_+\mid q_-(x)q_+(x)}}1.
\end{equation}
Imposing $\epsilon<\xi/2$, we make use of~\eqref{eq:leveldist} to evaluate the inner sum as
\[\sum_{\substack{x\in\cA(N)\\d_{\pm}\mid q_{\pm}(x)\\md_-d_+\mid q_-(x)q_+(x)}}1=|\cA(N)|\left(\prod_{p\mid d_+}\eta_+(p)\prod_{p\mid d_-}\eta_-(p)\prod_{p\mid m}(\eta_-(p)+\eta_+(p))+O\left(\frac{2^{\omega(m)}}{N^{\xi}}\right)\right).\]
Since $\lambda_m^{-}$ is supported on integers $m\leq y=N^{\theta}$ and $|\lambda_m^{-}|\leq 1$,  we see that the error term contributes $O(N^{-(\xi-2\theta)}|\cA(N)|)$ in~\eqref{eq:sieveinq} by bounding $2^{\omega(m)}$ trivially by $N^{\theta}$.
The assumption~\eqref{eq:assumeNT1} implies that there exists some constant $C_0$ such that
\[\prod_{p\leq z}\left(1-(\eta_-(p)+\eta_+(p))\right)=C_0(\log z)^{-(c_++c_-)}\left(1+O\left((\log z)^{-1}\right)\right)\]
holds for all $z$,
which satisfies~\cite[(6.47)]{IKbook} by picking $K$ to be larger than the implied constant.
Then using~\cite[(6.40)]{IKbook}, which is a more precise version of~\cite[(6.48)]{IKbook}, shows that
\begin{equation*}
 \sum_{\substack{m\\P^+(m)\leq z}}\lambda_m^-\mu(m)^2\prod_{p\mid m}(\eta_-(p)+\eta_+(p))
\geq 
 \left(1-e^{\beta-\frac{\theta}{\nu}}K^{10}\right)\prod_{p\leq z}\left(1-(\eta_-(p)+\eta_+(p))\right),
\end{equation*}
where $\beta=9(c_-+c_+)+1$.
Picking $\nu$ sufficiently small ensures that $1-e^{\beta-\frac{\theta}{\nu}}K^{10}\geq \frac{1}{2}$.
This allows us to lower bound~\eqref{eq:NT2} by
 \begin{multline}\label{eq:NT6}
 \gg |\cA(N)|\prod_{p\leq z}\left(1-(\eta_-(p)+\eta_+(p))\right)\sum_{\substack{(d_+,d_-)\\ d_+d_-\leq N^{\epsilon}\\P^+(d_+d_-)\leq z}}\mu(d_+d_-)^2\prod_{p\mid d_+}\alpha\eta_+(p)\prod_{p\mid d_-}\alpha^{-1}\eta_-(p)\\
 +O\left(\frac{|\cA(N)|}{N^{\xi-2\theta}}\sum_{\substack{(d_+,d_-)\\ d_+d_-\leq N^{\epsilon}}}\mu(d_+d_-)^2\alpha^{\omega(d_+)-\omega(d_-)}\right).
 \end{multline}
Taking $\epsilon<\xi-2\theta$ ensures that the error term is $o(|\cA(N)|)$.
As for the main term, we wish to remove the summation condition $d_+d_-\leq N^{\epsilon}$. 
Set $\delta(z)\coloneqq 1/\log z$ and assume that $N$ is sufficiently large such that $\delta(z)>0$.
We have
 \begin{align}
 &\sum_{\substack{(d_+,d_-)\\ d_+d_-> N^{\epsilon}\\P^+(d_+d_-)\leq z}}\mu(d_+d_-)^2\prod_{p\mid d_+}\alpha\eta_+(p)\prod_{p\mid d_-}\alpha^{-1}\eta_-(p)\notag\\
 &\leq
 \sum_{\substack{(d_+,d_-)\\P^+(d_+d_-)\leq z}}\mu(d_+d_-)^2\left(\frac{d_+d_-}{N^{\epsilon}}\right)^{\delta(z)}\prod_{p\mid d_+}\alpha\eta_+(p)\prod_{p\mid d_-}\alpha^{-1}\eta_-(p)\notag\\
 &\leq\frac{1}{N^{\epsilon\delta(z)}}
\exp\left(\sum_{p\leq z}p^{\delta(z)}\left(\alpha\eta_+(p)+\alpha^{-1}\eta_-(p)\right)\right).\label{eq:NT5}
 \end{align}
Since $e^t-1\leq 2t$ holds for all $0\leq t\leq 5/4$, we have
 $p^{\delta(z)}-1\leq 2\delta(z)\log p$ as long as $\delta(z)\log p\leq 5/4$.
 Therefore applying~\eqref{eq:assumeNT1}, we have
 \[\sum_{p\leq z}(p^{\delta(z)}-1)\left(\alpha\eta_+(p)+\alpha^{-1}\eta_-(p)\right)\leq 2\delta(z)\sum_{p\leq z}\log p\left(\alpha\eta_+(p)+\alpha^{-1}\eta_-(p)\right)\leq C_1 ,\]
 for some constant $C_1>0$ depending on $c_\pm$ and $\alpha$.
 This allows us to bound~\eqref{eq:NT5} by
\[
 \leq
 e^{C_1-\epsilon/\nu}
\exp\left(\sum_{p\leq z}\left(\alpha\eta_+(p)+\alpha^{-1}\eta_-(p)\right)\right).
\]
Using~\eqref{eq:assumeNT0}, we can find a constant $C_2>0$ such that
\[\prod_{p\leq z}\left(1+\alpha\eta_+(p)+\alpha^{-1}\eta_-(p)\right)
\geq C_2\exp\left(\sum_{p\leq z}\left(\alpha\eta_+(p)+\alpha^{-1}\eta_-(p)\right)\right).
\]
Choose $\nu$ small enough such that $C_3\coloneqq C_2-e^{C_1-\epsilon/\nu}>0$, so
 \[\sum_{\substack{(d_+,d_-)\\ d_+d_-\leq N^{\epsilon}\\P^+(d_+d_-)\leq z}}\mu(d_+d_-)^2\prod_{p\mid d_+}\alpha\eta_+(p)\prod_{p\mid d_-}\alpha^{-1}\eta_-(p)\geq C_3\exp\left(\sum_{p\leq z}\left(\alpha\eta_+(p)+\alpha^{-1}\eta_-(p)\right)\right).\]
Returning to~\eqref{eq:NT6}, and again using~\eqref{eq:assumeNT0}, we obtain
 \begin{equation*}
 \begin{aligned}
 &\sum_{x\in\cA(N)}\prod_p\alpha^{\rf_p(x)}\\
 &\gg|\cA(N)|\prod_{p\leq z}\left(1-(\eta_-(p)+\eta_+(p))\right)\cdot \exp\left(\sum_{p\leq z}\left(\alpha\eta_+(p)+\alpha^{-1}\eta_-(p)\right)\right)+o(|\cA(N)|)\\
 &\gg|\cA(N)|\exp\left(\sum_{p\leq z}\left(\alpha\eta_+(p)+\alpha^{-1}\eta_-(p)-(\eta_-(p)+\eta_+(p))\right)\right)+o(|\cA(N)|). 
 \end{aligned}
 \end{equation*}
 Plugging in~\eqref{eq:assumeNT1} yields the result.
\end{proof}

We note that an upper bound of the same order of magnitude as the lower bound can be obtained by an argument similar to that in~\cite[Section 3]{MR4870247}. However, we omit this here, as we are mainly interested in the distribution of the $\ell$-Selmer group, for which the Tamagawa ratio only yields a lower bound on its size.

\subsection{Tail bounds}
\begin{theorem}\label{theorem:tailbound}
 Keep the assumptions of Theorem~\ref{theorem:NT}. Then for any $A>0$ there exist $\delta(A)>0$ such that
\[\left|\left\{x\in\cA(N):\sum_{p}\rf_p(x)\geq A\log\log N\right\}\right|\gg_A|\cA(N)|(\log N)^{-\delta(A)}\]
holds for all $N\geq 2$.
\end{theorem}
\begin{proof}
Take $\alpha=1+(A+c_-)/c_+$. By Theorem~\ref{theorem:NT}, we have the lower bound \[\sum_{x\in\cA(N)}\prod_p\alpha^{\rf_p(x)}\geq C|\cA(N)|(\log N)^{(\alpha-1)c_+-(1-\alpha^{-1})c_-}\] for some constant $C>0$. Clearly $(\alpha-1)c_+-(1-\alpha^{-1})c_->(\alpha-1)c_+-c_-=A$, so
\[((\alpha-1)c_+-(1-\alpha^{-1})c_-)\log\log N+\log C\geq A\log\log N+\log 2\]
holds for all sufficiently large $N$.
Setting $\kappa\coloneqq A\log\log N/\log \alpha $,  we deduce that \[\sum_{x\in\cA(N)}\prod_p\alpha^{\rf_p(x)}\geq 2|\cA(N)|\alpha^{\kappa}\] for all sufficiently large $N$,
so 
\[\frac{1}{2}\sum_{x\in\cA(N)}\prod_p\alpha^{\rf_p(x)}\leq \sum_{x\in\cA(N)}\left(\prod_p\alpha^{\rf_p(x)}-\alpha^\kappa\right) \leq \sum_{x\in\cA(N)}\prod_p\alpha^{\rf_p(x)}\mathbf{1}_{\sum_{p}\rf_p(x)\geq \kappa}.\]
By Cauchy--Schwarz inequality, we have 
\[\sum_{x\in\cA(N)}\prod_p\alpha^{\rf_p(x)}\mathbf{1}_{\sum_{p}\rf_p(x)\geq \kappa}
 \leq 
\left(\sum_{x\in\cA(N)}\prod_p\alpha^{2\rf_p(x)}\right)^{\frac{1}{2}}
\left( \sum_{x\in\cA(N)}\mathbf{1}_{\sum_{p}\rf_p(x)\geq \kappa}\right)^{\frac{1}{2}}.
\]
Rearranging, we deduce that
\[ \left|\left\{x\in\cA(N):\sum_{p}\rf_p(x)\geq \kappa\right\}\right|=\sum_{x\in\cA(N)}\mathbf{1}_{\sum_{p}\rf_p(x)\geq \kappa}\geq 
 \frac{\left( \frac{1}{2}\sum_{x\in\cA(N)}\prod_p\alpha^{\rf_p(x)}\right)^{2}}{
\sum_{x\in\cA(N)}\prod_p\alpha^{2\rf_p(x)}}
.
\]
Using the upper bound in Theorem~\ref{theorem:NT} to bound the denominator and the lower bound in Theorem~\ref{theorem:NT} to bound the numerator, we have
\[\left|\left\{x\in\cA(N):\sum_{p}\rf_p(x)\geq A\log\log N\right\}\right|\gg_A|\cA(N)|(\log N)^{-\delta(A)},\]
where 
$\delta(A)
=(\alpha^2-1)c_+-2((\alpha-1)c_+-(1-\alpha^{-1})c_-)
=(\alpha-1)^2c_++2(1-\alpha^{-1})c_->0$.
\end{proof}

\subsection{The number of congruence classes}
In this section, we deduce some consequences of Chebotarev density theorem on the polynomials that determine the Tamagawa ratio of the elliptic curves in our families. Recall that $g\in\Z[t]$ and $B(x,y)=y^{3\varsigma}g(x/y^{\tau})$.

\begin{lemma}\label{lemma:equiprimes}
Suppose that $h\in\Z[t]$ has no repeated roots over $\overline{\Q}$ and is coprime with $g$ in $\Q[t]$.
Let $c$ be the number of irreducible factors of $h$ over $\Q$.
Let $h_1,\dots,h_c$ be the irreducible factors of $h$ over $\Q$ and let $L_i$ be the splitting field of $h_i$. Take $\alpha_i\in L_i$ to be a zero of $h_i$. Define 
\[
\theta_i\coloneqq\begin{cases}
 \hfil \frac{1}{2}&\text{if } 6g(\alpha_i)\notin L_i^2,\\ 
 \hfil 1&\text{if } 6g(\alpha_i)\in L_i^2,
\end{cases}\qquad \text{ and }\qquad \theta\coloneqq\theta_1+\cdots+\theta_c.
\]
Then there exists some constant $B_0$ and $B_1$ depending only on $g$ and $h$ such that
\[
\sum_{p\leq N}\frac{1}{p}\left|\left\{t\in\F_p: h(t)= 0\right\}\right|=c\log\log N+B_0 + O\left(\frac{1}{\log N}\right),
\]
and
\[
 \sum_{p\leq N}\frac{1}{p}\left|\left\{t\in\F_p: h(t)= 0,\ \leg{6g(t)}{p}=1\right\}\right|
 =
 \theta \log\log N+B_1 + O\left(\frac{1}{\log N}\right),\] 
where the implied constants depend only on $g$ and $h$.
\end{lemma}

\begin{proof}
The first part is~\cite[Proposition~3.10(e)]{SerreNp}, so we focus on the second part.

Let $d=\deg h$.
 Consider the set 
 \[\cQ=\left\{(\alpha_i,\sqrt{6g(\alpha_i)}),(\alpha_i,-\sqrt{6g(\alpha_i)}):i=1,\dots,d\right\},\]
 where $\alpha_1,\dots,\alpha_d$ are the distinct roots of $h(t)$ over $\overline{\Q}$. Let $K\coloneqq\Q(\{\alpha_i,\sqrt{6g(\alpha_i)}:i=1,\dots,d\})$. Notice that $K/\Q$ is Galois and $\Gal(K/\Q)$ permutes the elements in $\cQ$. Suppose that $p$ is a prime that does not divide the discriminant of $K$. Let $\sigma_p\coloneqq(K/\Q,p)$ denote the conjugacy class of Artin symbols in $\Gal(K/\Q)$ associated to the prime $p$. The elements in $K$ that are defined over $\F_p$ are precisely those fixed by $\sigma_p$. Since $h$ and $g$ have no common roots over $\overline{\Q}$, it is clear that $g(\alpha_i)\neq 0$. Let $c(\sigma)$ be the number of points in $\cQ$ that are fixed by $\sigma\in \Gal(K/\Q)$. For $p\nmid \Res(g,h)$, $h(t)\equiv g(t)\equiv 0\bmod p$ has no solution in $t\in\F_p$, so 
 \[\left|\left\{t\in\F_p: h(t)=0,\ \leg{6g(t)}{p}=1\right\}\right|=\frac{1}{2}c(\sigma_p).\]
 By the Chebotarev density theorem, we have
 \[
 \sum_{p\leq N}c(\sigma_p)=\frac{\sum_{\sigma\in\Gal(K/\Q)}c(\sigma)}{|\Gal(K/\Q)|}\cdot \frac{N}{\log N}+O\left(\frac{N}{(\log N)^2}\right).
 \]
 By Burnside's lemma, we find that
 \[
 \frac{\sum_{\sigma\in\Gal(K/\Q)}c(\sigma)}{|\Gal(K/\Q)|}
 \] is the number of $\Gal(K/\Q)$-orbits acting on $\cQ$, which equals $2\theta$. 
 The result follows by partial summation. 
\end{proof}

\begin{lemma}\label{lemma:equidistvD}
Let $D\in\Z[x,y]$ be a weighted homogeneous polynomial with weights $\tau,1$, and let $B(x,y)$ be as in~\eqref{eq:defAB}. Assume that $D$ and $B$ have no common zeros in $\PP^1_{(\tau,1)}(\overline{\Q})$.
Then there exists some constant $M$ such that
 \[
 \sum_{p\leq N}\frac{1}{p^2}\left|\left\{ (a,b)\in\F_p^2:
 \begin{array}{l}
 \gcd(a,b,p)=1,\\
 D(a,b)=0,\ \leg{6B(a,b)}{p}=1 
 \end{array}
 \right\}\right|
 =v(D)\log\log N+M+O\left(\frac{1}{\log N}\right),
\]
where $v(D)$ is as defined in Definition~\ref{def:uv}.
\end{lemma}

\begin{proof}
Consider the set
 \[
 \mathcal{D}(p)\coloneqq\left\{ (a,b)\in\F_p^2:
 \gcd(a,b,p)=1,\ 
 D(a,b)=0,\ \leg{6B(a,b)}{p}=1
 \right\}.
 \] 
By Lemma~\ref{lemma:resultantnonhom}, we may assume that $D$ and $B$ have no common zeros in $\PP^1_{(\tau,1)}(\F_p)$ after excluding finitely many primes.
We partition $\mathcal{D}(p)$ according to their image in $\PP^1_{(\tau,1)}(\F_p)$. 
For each $\alpha\in \PP^1_{(\tau,1)}(\F_p)$, fix a representative in $\F_p^2$, and take $\mathcal{C}(p)$ to be the set of all such representatives.
Recall that $D$ and $B$ are weighted homogeneous with weights $\tau,1$, and $B$ has weighted degree $3\varsigma$. Therefore
\begin{align*}
 |\mathcal{D}(p)|&=\sum_{(a,b)\in\mathcal{C}(p)}\left|\left\{ \lambda\in\F_p^{\times}:
 D(a,b)=0,\ \leg{6\lambda^{\varsigma}B(a,b)}{p}=1
 \right\}\right|\\ 
 &=\sum_{\substack{(a,b)\in\mathcal{C}(p)\\D(a,b)=0}}\left|\left\{ \lambda\in\F_p^{\times}: \leg{\lambda}{p}^{\varsigma}=\leg{6B(a,b)}{p}
 \right\}\right|.
\end{align*}
Since $\leg{\lambda}{p}$ is $1$ for $\frac{1}{2}(p-1)$ of $\lambda\in\F_p$ and is $-1$ for $\frac{1}{2}(p-1)$ of $\lambda\in\F_p$, we see that
\[\left|\left\{ \lambda\in\F_p^{\times}: \leg{\lambda}{p}^{\varsigma}=\leg{6B(a,b)}{p}
 \right\}\right|=
 \begin{cases}
 \hfil \frac{1}{2}(p-1)&\text{ if }\varsigma\text{ is odd},\\
 \hfil (p-1)\mathbf{1}_{\leg{6B(a,b)}{p}=1}&\text{ if }\varsigma\text{ is even}.
 \end{cases}\]
Therefore
\begin{align*}
 |\mathcal{D}(p)|
 &=
 \begin{cases}
 \hfil \frac{1}{2}(p-1)\left|\left\{[a:b]\in\PP^1_{(\tau,1)}(\F_p):D(a,b)=0\right\}\right|&\text{ if }\varsigma\text{ is odd},\\
 \hfil (p-1)\left|\left\{[a:b]\in\PP^1_{(\tau,1)}(\F_p):D(a,b)=0,\ \leg{6B(a,b)}{p}=1\right\}\right|&\text{ if }\varsigma\text{ is even}.
 \end{cases}
\end{align*}
Note that the conditions no longer depend on the choice of representatives of each class in $\PP^1_{(\tau,1)}(\F_p)$.
We may split up the sets according to whether or not $[a:b]=[1:0]$, so 
$|\mathcal{D}(p)|=C_0(p)+C_1(p)$, where
\begin{align*}
 C_0(p)
 &\coloneqq
 \begin{cases}
 \hfil \frac{1}{2}(p-1)\mathbf{1}_{p\mid D(1,0)}&\text{ if }\varsigma\text{ is odd},\\
 \hfil (p-1)\mathbf{1}_{p\mid D(1,0),\ \leg{6B(1,0)}{p}=1}&\text{ if }\varsigma\text{ is even},
 \end{cases}\\
 C_1(p)
 &\coloneqq
 \begin{cases}
 \hfil \frac{1}{2}(p-1)\left|\left\{t\in\F_p:D(t,1)=0\right\}\right|&\text{ if }\varsigma\text{ is odd},\\
 \hfil (p-1)\left|\left\{t\in\F_p:D(t,1)=0,\ \leg{6B(t,1)}{p}=1\right\}\right|&\text{ if }\varsigma\text{ is even}.
 \end{cases}
\end{align*}

Observe that
\[
\sum_{p\leq N}\frac{C_0(p)}{p^2}
=c_0\mathbf{1}_{y\mid D}\log\log N+M_0+O\left(\frac{1}{\log N}\right),
\]
where $M_0$ is a constant and
\[
c_0=\begin{cases}
 \hfil \frac{1}{2}&\text{if }\varsigma\text{ is odd or } 6B(1,0)\notin\Q^2,\\
 \hfil 1&\text{if }\varsigma\text{ is even and }6B(1,0)\in\Q^2.
\end{cases}
\]
Finally apply Lemma~\ref{lemma:equiprimes} to the contribution from $C_1(p)$ yields
\[
\sum_{p\leq N}\frac{C_1(p)}{p^2}
=c_1\log\log N+M_1+O\left(\frac{1}{\log N}\right),
\]
where $M_1$ is a constant and $c_1$ is the number of distinct irreducible factors of $D(t,1)$ in $\Q[t]$ when $\varsigma$ is odd, and $\theta$ as defined in Lemma~\ref{lemma:equiprimes} when $\varsigma$ is even. We can readily verify that $c_0\mathbf{1}_{y\mid D}+c_1=v(D)$. This completes the proof.
\end{proof}

\subsection{Sandwiching the Tamagawa ratio}
In this section, we define a sequence of random variables $(\rY^{\circ}_p)_p$ indexed by primes that approximates the Tamagawa ratio of the elliptic curves in our families. We then introduce two variants $(\rY^+_p)_p$ and $(\rY^-_p)_p$ of $(\rY^{\circ}_p)_p$, which are better suited to applications of the central limit theorem from Section~\ref{sec:central}.
 For each prime $p$, define $\rY^{\circ}_p:\F_p^2\rightarrow\{-1,0,1\}$ as follows.
 \begin{itemize} 
 \item If $\ell\geq 3$ and $p\nmid \gcd(A(a,b),B(a,b))\gcd(D_+(a,b),D_-(a,b))$, set
 \[
 \rY^{\circ}_p(a,b)=\begin{cases}
 1 &\text{ if } p\mid D_+(a,b) \text{ and }\leg{6B(a,b)}{p}=1,\\
 -1 &\text{ if } p\mid D_-(a,b) \text{ and }\leg{6B(a,b)}{p}=1.
 \end{cases}
 \]
\item If $\ell=2$ and $p\nmid \gcd(A(a,b),B(a,b))\gcd(D_+(a,b),D_-(a,b))$ set
 \[
 \rY^{\circ}_p(a,b)=\begin{cases}
 1 &\text{ if } p\mid D^{(1)}_+(a,b)\text{ or }p\mid D^{(2)}_+(a,b)\text{ and }\leg{6B(a,b)}{p}=1,\\
 -1 &\text{ if } p\mid D^{(1)}_-(a,b)\text{ or }p\mid D^{(2)}_-(a,b)\text{ and }\leg{6B(a,b)}{p}=1.
 \end{cases}
 \]
 \end{itemize}
 In any other case set $\rY^{\circ}_p(a,b)=0$.
\begin{lemma}\label{lemma:sandiwch}
Let $s\coloneqq\gcd(f^3,g^2)$. We impose the following conditions. 
\begin{itemize}
 \item 
If $\delta=1$ and $\ell\geq 5$, assume that
\begin{itemize}
 \item if $\deg s>0$, then $\upsilon=\tau=1$ and $s=k^r$, where $k$ is a squarefree polynomial in $\Q[t]$ and $r\in\{2,3,4,6\}$;
 \item $\deg s< \frac{24m}{2+m}$;
 \item if $\mult_k f=2$ and $\mult_k g=3$, then $\mult_k(4f^3+27g^2)=6$.
\end{itemize}
\item If $\ell\in\{2,3\}$, assume that $\deg s=0$. 
\end{itemize}
 There exist
 $\rY^+_p,\rY^-_p:(\Z/p^2\Z)^2\rightarrow \{-1,0,1\}$ such that 
 \begin{enumerate}
 \item Other than finitely many $p$, we have 
 $\rY^-_p(a,b)\leq \rY^{\circ}_p(a,b)\leq \rY^+_p(a,b)$
 for every $(a,b)\in(\Z/p^2\Z)^2$;
 \item Other than finitely many $p$, we have 
 \[\ell^{\rY_p^{-}(a,b)}\leq \frac{c_p(E_{a,b}')}{c_p(E_{a,b})}\leq \ell^{\rY_p^{+}(a,b)}\]
 for all $(a,b)\in\cT_{\upsilon,\tau}$ such that $\Disc(E_{a,b})\neq 0$;
 \item 
 $\left|\left\{(a,b)\in(\Z/p^2\Z)^2:\rY^-_p(a,b)\neq \rY^+_p(a,b)\right\}\right|\ll p^2$
 uniformly for all $p$;
 \item Each $\rY_p^\pm$ is constant on the set
 $\{(a,b)\in (\Z/p^2\Z)^2: p\mid \gcd(A(a,b),B(a,b)\}$;
 \item If $m=\delta=1$, then $\rY_p^-(a,b)=\rY_p^+(a,b)=0$ whenever $p\mid \Lambda$, where $\Lambda$ is defined in \eqref{def:Eps}; 
 \item If $\upsilon=\tau=1$, then $\rY_p^-(a,b)=\rY_p^+(a,b)=0$ whenever $p\mid \gcd(A(a,b),B(a,b))$.
 \end{enumerate}
\end{lemma}
\begin{proof}
Recall the factorisation of $\Delta$ given in~\eqref{eq:defDpm}.
First, consider the case when $\deg s=0$, so $A,B$ have no common factors.
Let $R$ be the product of all distinct irreducible factors of $D$.
Then $\rY^{\circ}_p$ agrees with the local Tamagawa ratios given in Lemma~\ref{lemma:tamratiocond} at all primes $p$ satisfying $p\nmid 6\ell\cdot \gcd(A(a,b),B(a,b))$, $ p^2\nmid R(a,b)$, and $v_p(c)=v_p(c')=0$.
For all sufficiently large $p$, take
 \begin{equation}\label{eq:defrYpm}
 \rY_p^\pm(a,b)=\begin{cases}
 \hfil\rY^{\circ}_p(a,b)&\text{if }p\nmid \gcd(A(a,b),B(a,b))\text{ and } p^2\nmid R(a,b),\\
 \hfil\pm1&\text{if }p\mid \gcd(A(a,b),B(a,b))\text{ or } p^2\mid R(a,b).
 \end{cases}
 \end{equation}
 Since $A,B$ are coprime, if $p\mid \gcd(A(a,b),B(a,b))$ then either $p\mid \Res(A,B)$ or $p\mid \gcd(a,b)$, so this contributes $O(p^2)$ many classes $(a,b)\bmod p^2$.
 Applying Lemma~\ref{lemma:rootbound} to $R$ shows that the number of $(a,b)\bmod p^2$ such that $p^2\mid D(a,b)$ is $O(p^2)$.
 If $\upsilon=\tau=1$, then by considering $\Res(A,B)$, we see that there are only finitely many primes such that such that $p\mid \gcd(A(a,b),B(a,b))$ for some $(a,b)\in\cT_{1,1}$, so we may set $\rY_p^{\pm}(a,b)=0$ when $p\mid \gcd(A(a,b),B(a,b))$ holds, while still requiring that \eqref{eq:defrYpm} holds for sufficiently large primes.
 This completes the proof in the case $\deg s=0$. Similarly, when $m=\delta=1$, we can set $\rY_p^-(a,b)=\rY_p^+(a,b)=0$ whenever $p\mid \Lambda$ since this only affects finitely many primes.

It remains to handle the case when $\deg s>0$. Under our restrictions, $\upsilon=\tau=1$ and $\ell\neq 2,3$. We check that the assumptions of Lemma~\ref{lemma:tamratiocond} are satisfied for all $(a,b)\in\cT_{1,1}$ and all primes $p$ outside of a finite set.
 Let $i\coloneqq \mult_k f$ and $j\coloneqq\mult_k g$. Let $K(x,y)=b^{\deg k}k(x/y)$ and write $A=K^iA'$ and $B=K^jB'$, so $K,A',B'$ are pairwise coprime. 
Assume that $p\nmid 6\ell\Res(D_+,D_-)\Res(K,A'B')\Res(A',B')$ and $v_p(c)=v_p(c')=0$. 
If $p\mid \gcd(A(a,b),B(a,b))$, then since $p\nmid \Res(K,A'B')\Res(A',B')$, this forces $p\mid K(a,b)$, and we further have $v_p(A(a,b))=iv_p(K(a,b))$ and $v_p(B(a,b))=jv_p(K(a,b))$. If $r=\min\{3i,2j\}\in\{2,3,4\}$, then it is clear that $3i\neq 2j$ so $3v_p(A(a,b))\neq 2v_p(B(a,b))$. If $r=6$, then our assumptions implies that $3i\neq 2j$ or $3i=2j=\mult_K \Delta$, so either $3v_p(A(a,b))\neq 2v_p(B(a,b))$ or $3v_p(A(a,b))=2v_p(B(a,b))=v_p(\Delta(a,b))$ 
In the case when $6=3i=2j=\mult_K \Delta$, remove the prime divisors of $\Res(K^r,\Delta/K^r)$. Therefore we may apply Lemma~\ref{lemma:tamratiocond} to $(a,b)\in\cT_{1,1}$. In particular, for any $p\mid \gcd(A(a,b),B(a,b))$ not in the excluded set, we have $p\nmid D_+(a,b)D_-(a,b)$ and Lemma~\ref{lemma:tamratiocond} shows that the local Tamagawa ratio is $1$, which agrees with $\rY^{\circ}_p(a,b)=0$. 
Other than finitely many $p$, Lemma~\ref{lemma:tamratiocond} then shows that $c_p(E_{a,b}')/c_p(E_{a,b})= \rY^{\circ}_p(a,b)$ holds for all $(a,b)\in\cT_{1,1}$. Therefore we may simply take $\rY_p^-(a,b)=\rY_p^+(a,b)=\rY^{\circ}_p(a,b)$ outside of a finite set of primes.
\end{proof}

\begin{lemma}\label{lemma:sandwichinput}
Take either $\upsilon=\tau=1$ or $\delta=m=1$. Under the assumptions of Lemma~\ref{lemma:sandiwch}, 
let $(\rY_p)_p$ be either $(\rY_p^{-})_p$ or $(\rY_p^{+})_p$.
 Then there exists a constant $\xi>0$ and a probability measure $\mathbf{P}$ on $\prod_p(\Z/p^2\Z)^2$ such that for any squarefree positive integer $q$ and $(\kappa_p)_{p\mid q}\in\{\pm 1\}^{\omega(q)}$, we have
\begin{equation}\label{eq:probassumpver}
 \frac{\left|\left\{(a,b)\in \cA^{\delta}_{\upsilon}(N):
(\rY_p(a,b))_{p\mid q}=(\kappa_p)_{p\mid q} \right\}\right|}{|\cA^{\delta}_{\upsilon}(N)|}=
 \prod_{p\mid q}\mathbf{P}\left(\rY_p=\kappa_p\right)+O\left(q^{4}N^{-\xi}\right)
\end{equation}
 for all sufficiently large $N$.
 Moreover
\begin{equation}\label{eq:musig}
\sum_{p\leq N}\mathbf{P}(\rY_p=\pm 1)=c_{\pm}\log \log N+M_{\pm}+O\left(\frac{1}{\log N}\right),
\end{equation}
for some constants $M_{\pm}$ and $c_{\pm}$ defined in~\eqref{eq:cpmdef}.
\end{lemma}

\begin{proof}
First consider the case when $\upsilon=\tau=1$. Apply Proposition~\ref{prop:ratio} if $\delta=1$, and Proposition~\ref{prop:easySequi} if $\delta=0$. There exists $\lambda(p), \epsilon>0$ such that for any $\alpha=[a:b]\in\PP^1(\Z/q\Z)$ with $\gcd(a,b,q)=\gcd(A(a,b),B(a,b),q)=1$, we have 
\begin{equation}\label{eq:firstcaseprob}
 \frac{\left|\cA^{\delta}_{1}(N, \alpha,q^2)\right|}{\left|\cA^{\delta}_{1}(N)\right|}
=
\frac{1}{|\PP^1(\Z/q^2\Z)|}\prod_{p\mid q}\lambda(p)+O(N^{-\epsilon})
\text{
and
}\lambda(p)=1+O\left(\frac{1}{p}\right).
\end{equation}
Note that $\rY_p$ factors through $\mathbb{P}^1(\Z/p^2\Z)$.
Take $\mathbf{P}$ to be a measure on $\prod_p(\Z/p^2\Z)^2$ such that $\mathbf{P}([a:b]\equiv \alpha,\ \gcd(a,b,p)=1)=\frac{1}{|\PP^1(\Z/p^2\Z)|}\lambda(p)$ given any $\alpha\in\mathbb{P}^1(\Z/p^2\Z)$. Note that we can freely fix the probability of elements $(a,b)\in\prod_p(\Z/p^2\Z)^2$ such that $p\mid \gcd(a,b)$.
Summing~\eqref{eq:firstcaseprob} over all $\alpha\in\mathbb{P}^1(\Z/q^2\Z)$ such that $(\rY_p(\alpha))_{p\mid q}=(\kappa_p)_{p\mid q}$, we see that the main term matches with the main term of~\eqref{eq:probassumpver}. The number of such $\alpha$ is trivially bounded by $q^2$, so the error term becomes $O(q^2N^{-\epsilon})$. Since Lemma~\ref{lemma:sandiwch} implies that $\rY_p(a,b)=0$ whenever $p\mid \gcd(A(a,b),B(a,b))$, the condition $p\nmid \gcd(A(a,b),B(a,b))$ does not affect the sum.

Now assume instead that $\upsilon=\tau=1$ does not hold, so $f$ and $g$ are coprime by the assumptions of Lemma~\ref{lemma:sandiwch}, and $\delta=m=1$. To verify \eqref{eq:probassumpver}, we apply Proposition~\ref{prop:nonhomprob}. It follows that, there exists $\lambda(p), \epsilon>0$ such that for any $(a,b)\in(\Z/q\Z)^2$ with $\gcd(a,b,q)=\gcd(q,\Lambda)=1$, we have
\begin{equation}\label{eq:concin}
 \frac{\left|\cA^1_{\upsilon}(N, (a,b),q^2)\right|}{\left|\cA^1_{\upsilon}(N)\right|}
=
\frac{1}{q^4}\prod_{p\mid q}\lambda(p)+O(N^{-\epsilon})
\text{
and
}\lambda(p)=1+O\left(\frac{1}{p}\right).
\end{equation}
Take $\mathbf{P}$ to be a measure on $(\Z/p^2\Z)^2$ such that $\mathbf{P}((a,b)\equiv (a_0,b_0))=\frac{1}{p^4}\lambda(p)$ given any $(a_0,b_0)\in(\Z/p^2\Z)^2$ such that $p\nmid \gcd(A(a_0,b_0),B(a_0,b_0))$. 
By Lemma~\ref{lemma:sandiwch}, we see that $\rY_p$ is constant on the set
 $\{(a,b)\in (\Z/p^2\Z)^2: p\mid \gcd(A(a,b),B(a,b)\}$ and $0$ when $p\mid \Lambda$. Since we only need to verify \eqref{eq:probassumpver} when $\kappa_p$ are all non-zero, we only need to consider the case when $q$ is coprime to $\Lambda$.
 Suppose $q=q_0q_1$ and let $a_0,b_0$ be integers such that $\gcd(q_0,A(a_0,b_0), B(a_0,b_0))=1$.
 We use \eqref{eq:concin} to evaluate
 \begin{align*}
& \frac{1}{\left|\cA^1_{\upsilon}(N)\right|} \left| \left\{(a,b)\in\cA^1_{\upsilon}(N):
 \begin{array}{l}
 (a,b)\equiv (a_0,b_0)\bmod q_0^2\\
 q_1\mid \gcd(A(a,b), B(a,b))\\
 \end{array}\right\}\right|\\
  &= \sum_{d\mid q_1}
  \frac{\mu(d)}{\left|\cA^1_{\upsilon}(N)\right|} \left| \left\{(a,b)\in\cA^1_{\upsilon}(N):
 \begin{array}{l}
 (a,b)\equiv (a_0,b_0)\bmod q_0^2\\
 \gcd(d,A(a,b), B(a,b))=1
 \end{array}\right\}\right|\\
 &=\sum_{d\mid q_1}
  \mu(d)\sum_{\substack{(a_1,b_1)\bmod q_0^2d^2\\
 (a_1,b_1)\equiv (a_0,b_0)\bmod q_0^2\\
 \gcd(d,A(a_1,b_1),B(a_1,b_1))=1}} \frac{|\cA^1_{\upsilon}(N, (a_1,b_1),q_0^2d^2)|}{\left|\cA^1_{\upsilon}(N)\right|}\\
 &=
\sum_{d\mid q_1}\mu(d)\sum_{\substack{(a_1,b_1)\bmod d\\
\gcd(d,A(a_1,b_1),B(a_1,b_1))=1}}\frac{1}{(q_0d)^4}\prod_{p\mid q_0d}\lambda(p)+O\left(2^{\omega(q_1)}q_1^4N^{-\epsilon}\right)\\
&=
\prod_{p\mid q_0}\frac{\lambda(p)}{p^4}\prod_{p\mid q_1}\left(1-\frac{\lambda(p)}{p^4}\sum_{\substack{(a_1,b_1)\bmod p\\
p\nmid \gcd(A(a_1,b_1),B(a_1,b_1))}}1\right) +O\left(2^{\omega(q_1)}q_1^4N^{-\epsilon}\right).
 \end{align*}
Summing over all $(a_0,b_0)\bmod q_0^2$ and $(q_0,q_1)$ such that $(\rY_p(a,b))_{p\mid q}=(\kappa_p)_{p\mid q}$ gives the required estimate \eqref{eq:probassumpver} on taking $\xi<\epsilon$ and noting that the estimate is trivial if $q_1\gg N$.

Finally to prove~\eqref{eq:musig}, note that by Lemma~\ref{lemma:sandiwch}, other than finitely many $p$,
\[\mathbf{P}(\rY_p=\pm 1)=\mathbf{P}(\rY^{\circ}_p=\pm 1)+O\left(\frac{1}{p^2}\right).\]
Lemma~\ref{lemma:equidistvD} provides us with
\begin{align*}
 \sum_{p\leq N}\mathbf{P}(\rY^{\circ}_p=\pm 1)&=
\sum_{p\leq N}\frac{\lambda(p)}{p^2}\left|\left\{(a,b)\in\F_p^2:\rY^{\circ}(a,b)=\pm 1\right\}\right|\\&=c_{\pm}\log \log N+M_{\pm}+O\left(\frac{1}{\log N}\right),
\end{align*}
 where $M_-$ and $M_-$ are constants. Combining with the previous estimate yields~\eqref{eq:musig}.
 \end{proof}

\subsection{Main theorem}
We are finally ready to prove some results on the distribution of the Tamagawa ratio by combining the work of the previous sections.
Let $f,g,f',g'\in\Z[t]$ and fix a prime $\ell$. 
Let 
\[
\cE:y^2=x^3+f(t)x+g(t),\text{ and }\quad\cE'=y^2=x^3+f'(t)x+g'(t)
\]
be elliptic curves over $\Q(t)$ with a degree $\ell$ isogeny $\phi:\cE\rightarrow\cE'$ over $\Q(t)$.
We may specialise to $\cE_t,\ t\in\Q$ such that $\Disc(\cE_t)\neq 0$ and take its quadratic twist $\cE_t^d$ by some $d\in\Q^{\times}$, so $\phi$ specialises to a degree $\ell$ isogeny $\cE_t^d\rightarrow (\cE'_t)^d$. 

Recall the multisets $\cF(N), \cG(N), \cF_0(N)$ defined in~\eqref{def:cFG} and~\eqref{def:cF0}.

\begin{definition}[admissible families]\label{def:add}
Let $f,g\in\Z[t]$.
Suppose that $\cE:y^2=x^3+f(t)x+g(t)$ admits a prime degree $\ell$ isogeny $\phi$ over $\Q(t)$. We say that $\cS=\cup_N \cS(N)$, a subset of either $\cF$ and $\cG$, is an admissible family with respect to $\phi$ if one of~\ref{Anormal},~\ref{Atwists},~\ref{Ah1},~\ref{Ah0} holds. 
\begin{enumerate}[label=(A\arabic*)]
\item \label{Anormal}
We require $f$ and $g$ to be coprime in $\Q[t]$, and 
\[
 \max\left\{\frac{1}{4}\deg f,\frac{1}{6}\deg g\right\}=\frac{m}{\tau}\text{ for some }\tau,m\in\Z_{\geq 1}\text{ and one of $m,\tau$ is $1$}.
\]
We require that $\cS(N)\subseteq \cF(N)$ and there exists $\xi>0$ such that
$|\cS(N)|=(1+O(N^{-\xi})) |\cF(N)|$ for all sufficiently large $N$.
\item \label{Atwists}
We require $f$ and $g$ to be coprime in $\Q[t]$, and 
\[
 \max\left\{\frac{1}{2}\deg f,\frac{1}{3}\deg g\right\}=\frac{1}{\tau}\text{ for some }\tau\in\Z_{\geq 1}.
\]
We require that $\cS(N)\subseteq \cG(N)$ and there exists $\xi>0$ such that
$|\cS(N)|=(1+O(N^{-\xi})) |\cG(N)|$ for all sufficiently large $N$.
 \item \label{Ah1}
 We require $f$ and $g$ to have no common real roots and
\[
 m\coloneqq\max\left\{\frac{1}{4}\deg f,\frac{1}{6}\deg g\right\}\in\Z_{\geq 1}.
\]
We require that $\ell\geq 5$.
Require that $s\coloneqq \gcd(f^3,g^2)$ satisfies
\begin{itemize}
 \item $4\leq \deg s<\frac{24m}{2+m}$;
 \item $s(t)=k(t)^r$, where $k(t)$ is a squarefree polynomial in $\Q[t]$ and $r\in\{2,3,4,6\}$;
 \item if $\mult_k f=2$ and $\mult_k g=3$, then $\mult_k(4f^3+27g^2)=6$.
\end{itemize}
We require that $\cS(N)\subseteq \cF(N)$ and there exists $\xi>0$ such that
$|\cS(N)|=(1+O(N^{-\xi})) |\cF(N)|$ for all sufficiently large $N$.
\item \label{Ah0}
 We require $f$ and $g$ to have no common real roots and
\[
 m\coloneqq\max\left\{\frac{1}{4}\deg f,\frac{1}{6}\deg g\right\}\in\Z_{\geq 1}.
\]
If $\ell\in\{2,3\}$, then also $f$ and $g$ are coprime. We require that $\cS(N)\subseteq \cF_0(N)$ and there exists $\xi>0$ such that
$|\cS(N)|=(1+O(N^{-\xi})) |\cF_0(N)|$ for all sufficiently large $N$.

\end{enumerate}
\end{definition}

The order of $|\cS(N)|$ and the parameters in~\eqref{eq:degfg} taken for each case are as follows.

 \begin{tabular}{llllll}
 \ref{Anormal} & $|\cS(N)|\asymp|\cA^1_1(N)|\asymp N^{\frac{\tau+1}{12m}}$, & $\delta=1$, & $\upsilon=1$, & $\varsigma=2m$;\\
 \ref{Atwists} & $|\cS(N)|\asymp|\cA^1_2(N)|\asymp N^{\frac{\tau+1}{6}}$, & $\delta=1$, & $\upsilon=2$, & $\varsigma=1$, & $m=1$; \\
 \ref{Ah1} & $|\cS(N)|\asymp|\cA^1_1(N)|\asymp N^{\frac{1}{6m}}$, & $\delta=1$, & $\upsilon=1$, & $\varsigma=2m$, & $\tau=1$;\\
 \ref{Ah0} & $|\cS(N)|\asymp|\cA^0_1(N)|\asymp N^{\frac{1}{6m}}$, & $\delta=0$, & $\upsilon=1$, & $\varsigma=2m$, & $\tau=1$. \\
 \end{tabular}

Recall the constants $c_\pm$ depending on the family $\cS$ defined in~\eqref{eq:cpmdef}.
Set
\begin{gather}
 (\mu,\sigma^2)\coloneqq
 \left(c_+-c_-,\ c_++c_-\right),\label{eq:musigdef}
\\
\rho(k)\coloneqq(\ell^k-1) c_+-(1-\ell^{-k})c_-.\label{eq:rhodef}
\end{gather}
\begin{theorem}\label{theorem:mainhom}
Suppose that $\cS$ is an admissible family with respect to $\phi$ as defined in Definition~\ref{def:add}. Let $\mu,\sigma,\rho$ be as in~\eqref{eq:musigdef} and~\eqref{eq:rhodef}. Given any $E\in \mathcal{S}$, let $r_{\phi}(E)$ be the logarithmic Selmer ratio defined in~\eqref{eq:selmerratiorankdef}. Let $k>0$ and $A>0$.
Then
\[
\left\{\frac{r_{\phi}(E)-\mu\log\log N}{\sigma\sqrt{\log\log N}}:E\in\cS(N)\right\}
\] converges in distribution to a standard Gaussian random variable as $N\rightarrow \infty$, and
\begin{equation} 
 \sum_{E\in\cS(N)}\ell^{k\cdot r_{\phi}(E)}\gg_k|\cS(N)|(\log N)^{\rho(k)}. \label{eq:avsel}
\end{equation}
There exists $\delta=\delta(A)>0$ such that
 \begin{equation} 
 \left|\left\{E\in\cS(N):r_{\phi}(E)\geq A\log\log N\right\}\right|\gg_A|\cS(N)|(\log N)^{-\delta}\label{eq:tailbd}.
\end{equation}
\end{theorem}

\begin{proof}
If $\cS$ is admissible under~\ref{Anormal} or~\ref{Ah1}, we work with $\cA^{1}_1$.
If $\cS$ is admissible under~\ref{Atwists}, we work with $\cA^{1}_2$.
If $\cS$ is admissible under~\ref{Ah0}, we work with $\cA^{0}_1$.
By Lemma~\ref{lemma:cpsel}, we have
\[
\ell^{r_{\phi}(E)}\asymp\prod_p \frac{c_p(E')}{c_p(E)}.
\]

To approximate the local Tamagawa ratios, take $(\rY^{\pm}_p)_p$ from Lemma~\ref{lemma:sandiwch}.
We will apply Theorem~\ref{theorem:EKver} to $\rY_p^{\pm}$ with the choice $Q=\exp((\log\log N)^{-\frac{1}{3}}\log N)$.
Lemma~\ref{lemma:sandwichinput} provides the estimate~\eqref{eq:probindest} required. This shows that 
\[
 \left\{\frac{\sum_{p\leq Q}\rY^{\pm}_p(a,b)-\mu(Q)}{\sigma(Q)}: (a,b)\in\cA^{\delta}_{\upsilon}(N)\right\}
 \]
converges in distribution to the standard Gaussian random variable as $N\rightarrow\infty$.
To compute $\mu(Q)$ and $\sigma(Q)$, we use~\eqref{eq:musig} to get
\[\mu(Q)=(c_+-c_-)\log\log Q+O(1)\quad \text{ and }\quad \sigma(Q)^2=(c_++c_-)\log\log Q+O(1).\]
Given $(a,b)\in\cA^{\delta}_{\upsilon}(N)$, since $\rY_p^{\pm}(a,b)\neq 0$ only if $p\mid D(a,b)$, and $D(a,b)\ll N^{O(1)}$, we have the bound
 \[
 |\{p>Q:\rY^{\pm}_p(a,b)\neq 0\}|\ll\frac{\log N}{\log Q}=(\log\log N)^{\frac{1}{3}}.\]
 Plugging in these estimates, we see that 
 \[\frac{\sum_{p}\rY^{+}_p(a,b)-(c_+-c_-)\log\log N}{\sqrt{(c_++c_-)\log\log N}}\quad\text{ and }\quad\frac{\sum_{p}\rY^{-}_p(a,b)-(c_+-c_-)\log\log N}{\sqrt{(c_++c_-)\log\log N}}\]
 both converge in distribution to the standard Gaussian.
 Now Lemma~\ref{lemma:sandiwch} tells us that 
 $r_{\phi}(E_{a,b})\leq \sum_p\rY_p^{+}(a,b)+O(1)$ and 
 $r_{\phi}(E_{a,b})\geq \sum_p\rY_p^{-}(a,b)+O(1)$ both hold for all $(a,b)\in\cT_{\upsilon,\tau}$ with $\Disc(E_{a,b})\neq 0$, so $(r_{\phi}(E_{a,b})-\mu\log\log N)/(\sigma\sqrt{\log\log N})$ also converges to the standard Gaussian.

Next we apply Theorem~\ref{theorem:NT} and
Theorem~\ref{theorem:tailbound} to $\rY_p^{-}$. To check that~\eqref{eq:leveldist} holds, Lemma~\ref{lemma:sandwichinput} implies that \eqref{eq:probassumpver} holds for some $\xi>0$, then for all $q\leq N^{\xi/8}$, the error term of \eqref{eq:probassumpver} can be trivially bounded by $O(N^{-\xi/2})$. This shows that \eqref{eq:leveldist} holds with $\xi/4$ in place of $\xi$.
Since $r_{\phi}(E_{a,b})\geq \sum_p\rY_p^{-}(a,b)+O(1)$, we conclude that 
\begin{equation*}
 \sum_{(a,b)\in\cA^{\delta}_{\upsilon}(N)}\ell^{k\cdot r_{\phi}(E_{a,b})}\gg|\cA^{\delta}_{\upsilon}(N)|(\log N)^{\rho(k)}, 
 \end{equation*}
 and 
\begin{equation*}
 \left|\left\{(a,b)\in\cA^{\delta}_{\upsilon}(N):r_{\phi}(E_{a,b})\geq c\log\log N\right\}\right|\gg|\cA^{\delta}_{\upsilon}(N)|(\log N)^{-\delta}
 \end{equation*}
for some positive constants $\delta$ and $A$. 

To transfer from $\cA^{\delta}_{\upsilon}(N)$ to the corresponding $\cF(N)$, or $\cG(N),\cF_0(N)$, it suffices to note that due to Lemma~\ref{lemma:removingsingular}, the proportion of $(a,b)\in \cA^{\delta}_{\upsilon}(N)$ such that $\Disc(E_{a,b})=0$ is $O(N^{-\epsilon})$ for some $\epsilon>0$. 
Since $\cS$ is admissible, its proportion is $1+O(N^{-\xi})$ inside the relevant $\cF(N)$, $\cG(N)$, or $\cF_0(N)$. Clearly this does not change the convergence in distribution, as well as~\eqref{eq:tailbd}, since the proportion of exceptions is $O(N^{-\xi})$.
As for~\eqref{eq:avsel}, observe that $r_{\phi}(E_{a,b})=o(\log N)$, and hence $\ell^{r_{\phi}(E_{a,b})}\ll N^{\xi/2}$, so dropping a subset of size $O(N^{-\xi})$ does not affect the order of magnitude of the lower bound.
\end{proof}

\section{Families of elliptic curves}\label{sec:fam}
The general strategy is to show that certain families are admissible under Definition~\ref{def:add}. Then use Velu's formula to compute the isogeny, and find $\mu, \sigma, \rho(1)$ from the expressions.

The following lemma shows that we may disregard the elliptic curves in $\cF$ and $\cG$ with $j$-invariant $0$ or $1728$, as long as $f$ and $g$ are both non-zero polynomials.
\begin{lemma}\label{lemma:cusps}
Assume that $f(t)$ and $g(t)$ are both non-zero polynomials. Then
\begin{gather}
 \left|\{E\in\cF(N):j(E)=0\text{ or }1728\}\right|\ll 1,\label{eq:exceptFj}\\
 \left| \{E\in\cG(N):j(E)=0\text{ or }1728\}\right|\ll N^{\frac{1}{6}},\label{eq:exceptGj}
\end{gather}
where the implied constants depend only on $f$ and $g$.
If $\cS$ is admissible, then $\{E\in \cS: j(E)\neq 0,1728\}$ is also admissible.
\end{lemma}
\begin{proof}
 If $j(\cE_t)=0$ or $1728$, then $t$ must satisfy either $f(t)=0$ or $g(t)=0$.
 Since $f(t)$ and $g(t)$ are both non-zero polynomials, there are only finitely many $t$ that is a root of either $f$ or $g$.
 Therefore~\eqref{eq:exceptFj} is immediate.
As for~\eqref{eq:exceptGj}, note that given $\cE_t$, the number of quadratic twists up to height $N$ is $O(N^{\frac{1}{6}})$. This shows~\eqref{eq:exceptGj}.

If $\cS$ is admissible of type~\ref{Anormal},~\ref{Ah0}, or~\ref{Ah1}, then the order of magnitude of $|\cS(N)|$ is a positive power of $N$, so~\eqref{eq:exceptFj} shows that $\{E\in \cS: j(E)\neq 0,1728\}$ is also
admissible.
If $\cS$ is admissible of type~\ref{Atwists}, then since $|\cS(N)|\asymp N^{\frac{\tau+1}{6}}$, the bound~\eqref{eq:exceptGj} is $O(N^{-\frac{\tau}{6}})|\cS(N)|$, so again $\{E\in \cS: j(E)\neq 0,1728\}$ is
admissible.
\end{proof}

\subsection{Torsion subgroups with exponent at least $4$}
By Mazur's Theorem~\cite[Theorem~8]{mazur}, the torsion subgroup $E_{\tors}(\Q)$ of $E(\Q)$ is isomorphic to one of $\Z/n\Z$, where $n=1,\dots,10,12$, and $\Z/2\Z\times \Z/2n\Z$, where $n=1,2,3,4$.

\begin{lemma}\label{lemma:controlfibretorsion}
Let $T=\Z/m\Z\times \Z/mn\Z$, where \[(m,n)\in\{(1,4),(1,5),\dots, (1,10),(1,12),(2,2),(2,3),(2,4)\}.\] Set $f(t)=-27 \alpha_T(t,1)$ and $g(t)=-54\beta_T(t,1)$, where $\alpha_T$ and $\beta_T$ are as given in~\cite[Tables 4 and 5]{barrios2022minimal}.
Other than finitely many $E\in \cF$, the size of the fibre of any $E\in \cF$ under the map
 \[\cF\rightarrow \cF/\cong_{\Q}\]
 equals to the number of embeddings $T\hookrightarrow E_{\tors}(\Q)$.
\end{lemma}
\begin{proof}
The non-cuspidal rational points of the modular curve $X_{1}(m,mn)$ parametrise the $\Q$-isomorphism classes of triples $(E,P,Q)$, where $E$ is an elliptic curve, and $P,Q\in E(\Q)$ are independent points of orders $m$ and $mn$. See for example~\cite[Section C.13]{arithmetic}. 

We exclude finitely many $E\in \cF$ corresponding to the cuspidal points of $X_{1}(m,mn)$. As described in~\cite[Section 2.3]{barrios2022minimal} (also~\cite[Table~3]{kubert1976universal}), the modular parametrisation results in a single parameter $t\in\PP^1(\Q)$, with the universal elliptic curve given by the equation $y^2=x^3+f(t)x+g(t)$ (\cite[Lemma~2.9]{barrios2022minimal}). 

Note that the set $\cF$ does not contain any degenerations as we have imposed the discriminant of the elliptic curve to be non-zero. Given any $\Q$-isomorphism class of elliptic curves in $\cF$, the number of allowable $P,Q$ is exactly the number of embeddings $T\hookrightarrow E_{\tors}(\Q)$. 
\end{proof}

We describe the general strategy for a family $\cS$ of all elliptic curves with given torsion subgroup $T$, when $T$ has exponent at least $4$. 
Take $f,g$ as in Lemma~\ref{lemma:controlfibretorsion}. Take $\ell$ to be a prime that divides the order of $T$, so $\cE(\Q(t))$ has a point of order $\ell$. Fix such a point, then take $\phi:\cE\rightarrow \cE'$ to be the isogeny with kernel generated by this point. 
We can readily check that $f,g$ satisfy the conditions in~\ref{Anormal}.
By~\cite[Theorem 1.2]{harron2017counting}, the number of elliptic curves in $\cS(N)$ with a torsion subgroup strictly containing $T$ is bounded by $O(N^{-\xi})$ for some $\xi>0$, so $\cS$ is admissible with respect to $\phi$ under~\ref{Anormal}.
Lemma~\ref{lemma:controlfibretorsion} shows that the same result holds when we look at $\Q$-isomorphism classes.

We first use Velu's formula~\cite{velu1971isogenies} to compute a model for $\cE'$, as well as the isogeny $\phi$. This information then allows us to compute the constants $\mu$, $\sigma$, and $\rho(1)$ as needed. The constants obtained are tabulated in Table~\ref{table:oddprime}, Table~\ref{table:twounique}, and Table~\ref{table:twochoice}. 

\subsubsection{Torsion subgroup $\Z/5\Z$}
As an example, we lay out the computations for $T=\Z/5\Z$ here. We obtain from~\cite[Tables 4 and 5]{barrios2022minimal} 
\begin{align*}
 f(t)&=-27 \alpha_T(t,1)=-27t^4 - 324t^3 - 378t^2 + 324t - 27,\\
 g(t)&=-54\beta_T(t,1)=54(t^4 + 18t^3 + 74t^2 - 18t + 1)(t^2 + 1). 
\end{align*}
In this case $(m,\tau)=(1,1)$. We obtain from~\eqref{eq:defAB}, 
 \begin{align*}
 A(a,b)&=-27 (a^4 + 12 a^3 b + 14 a^2 b^2 - 12 a b^3 + b^4),\\
 B(a,b)&=54(a^4 + 18a^3b + 74a^2b^2 - 18ab^3 + b^4)(a^2 + b^2). 
\end{align*}
From which we can compute 
 \[
 \Delta(a,b)=4A(x,y)^3+27B(x,y)^2=2^8\cdot 3^{12}(a^2 + 11ab - a)a^5b^5.
 \]
 The four $5$-torsion points of $E_{a,b}$ have coordinates $(3(a^2+6ab+b^2),\pm 108ab^2)$ and $(3(a^2-6ab+b^2),\pm 108a^2b)$. Applying Velu's formula, $E_{a,b}$ is $5$-isogenous to an elliptic curve $E'_{a,b}$ with
 \begin{align*}
 A'(a,b)&=-27 (a^4 - 228 a^3 b + 494 a^2 b^2 + 228 a b^3 + b^4),\\
 B'(a,b)&=54(a^4 + 522a^3b - 10006a^2b^2 - 522ab^3 + b^4)(a^2 + b^2),\\
 \Delta'(a,b)&=2^8\cdot 3^{12}(a^2 + 11ab - b^2)^5ab.
\end{align*}
 Therefore we can take $D_+(a,b)=a^2 + 11ab - b^2$ and $D_-(a,b)=ab$, which gives $u_+=1$ and $u_-=2$. We have $6B(1,0)=6B(0,1)=324=18^2$, so $v_-=2$. The roots of $D_+(t,1)$ are $t_{\pm}=(-11\pm5\sqrt{5})/2$ and \[
 6g(t_\pm)=2^2\cdot 3^4\cdot 5^7(\pm 682 \sqrt{5} - 1525)
 \]
 is not a square in $\Q(\sqrt{5})$, thus $v_+=1/2$.
 This gives $\mu=-3/2$ and $\sigma^2=5/2$.
 
\subsubsection{Torsion subgroup $\Z/4\Z$}
The family of elliptic curves with torsion subgroup $\Z/4\Z$ satisfies~\ref{Anormal} with 
 \[
 f(t)=-27(16t^2+16t+1)
 \quad \text{ and }\quad
 g(t)=-54(8t+1)(8t^2-16t-1).
 \]
 Let $(m,\tau)=(1,2)$, so
 \[
 A(a,b)=-27(16a^2+16ab^2+b^4)\quad \text{ and }\quad
 B(a,b)=-54(8a+b^2)(8a^2-16ab^2-b^4).
 \]
 See also~\cite[Section~3.1]{harron2017counting}. 
 We have
 \[
 \Delta(a,b)=-2^8\cdot 3^{12}a^4b^2(16a+b^2),
 \quad \text{ and }\quad
 \Delta'(a,b)=-2^8\cdot 3^{12}a^2b^4(16a+b^2)^2,
 \]
 so $u_+^{(1)}=1$, $u_+^{(2)}=1$, $u_-^{(1)}=0$, $u_-^{(2)}=1$, $v_+^{(2)}=1/2$, $v_-^{(2)}=1$. Then $\mu=1/2$ and $\sigma^2=5/2$.

\subsection{Torsion subgroups with exponent less than $4$}
\subsubsection{Torsion subgroup $\Z/2\Z$}
We will show that the family of elliptic curves with torsion subgroup $\Z/2\Z$ is of type~\ref{Atwists}.
Take
 \[
 f(t)=t \quad \text{ and }\quad
 g(t)=t+1.
 \]
 Clearly $\cE:y^2=x^3+tx+(t+1)$ has a $2$-torsion point given by $(-1,0)$.
Take $\tau=2$, so
 \[
 A(a,b)=a \quad \text{ and }\quad B(a,b)=b^3+ab.
 \]
 By~\cite[Lemma~5.1]{harron2017counting}, there is a one-to-one correspondence between $\{(E/\Q, \Z/2\Z\hookrightarrow E^{\tors}(\Q))\}/\cong_\Q$ and curves $\{(E_{a,b} : (a,b)\in\cT_{2,2},\ \Delta(a,b)\neq 0\}$. By~\cite[Theorem 1.2]{harron2017counting}, the number of elliptic curves with a torsion subgroup strictly containing $\Z/2\Z$ up to height $N$ is bounded by $O(N^{-\xi})$ for some $\xi>0$, so the family of elliptic curves with torsion subgroup $\Z/2\Z$ is admissible of type~\ref{Atwists} with respect to the $2$-isogeny with kernel generated by the $2$-torsion point.

 We have
 \[
 \Delta(a,b)=(a + 3 b^2)^2 (4 a + 3 b^2),
 \]
 and, by Velu's formula (notice $(-b,0)$ is a $2$-torsion point), $E_{a,b}$ is $2$-isogenous to an elliptic curve $E'_{a,b}$ with
 \[
 \Delta'(a,b)=-16(a + 3 b^2) (4 a + 3 b^2)^2.
 \]
 We may take $D_+(a,b)=4a + 3 b^2$ and $D_-(a,b)=a + 3b^2$, so $u_+^{(1)}=u_-^{(1)}=1$ and $u_+^{(2)}=u_-^{(2)}=0$, and hence $\mu=0$ and $\sigma^2=2$.

\subsubsection{Torsion subgroup $\Z/2\Z\times\Z/2\Z$}
To show that the family of elliptic curves with torsion subgroup $\Z/2\Z\times\Z/2\Z$ is of type~\ref{Atwists}, take
 \[
 f(t)=-27(t^2-t+1) \quad \text{ and }\quad
 g(t)=-27(2t^3-3t^2-3t+2).
 \]
 \begin{lemma}\label{lemma:full}
Take the above choice of $f$ and $g$.
If $E$ is an elliptic curve such that $\Z/2\Z\times\Z/2\Z\hookrightarrow E_{\tors}(\Q)$ and $j(E)\neq 0,1728$, then $E$ is $\Q$-isomorphic to an element in $\cG$.
The size of the fibre of any $E\in\cG$ with $j(E)\neq 0,1728$ under the map $\cG\to\cG/\cong_{\Q}$ 
is $6$.
\end{lemma}
\begin{proof}
If $E$ is a rational elliptic curve with $E(\Q)[2]\cong\Z/2\Z\times \Z/2\Z$, then it has a Weierstrass model of the form $y^2=x(x-a)(x-b)$ for $a,b\in \Z$. This is a quadratic twist by $B$ of the curve $y^2=x(x-t)(x-1)$, where $t=a/b$. Employing the change of variables $(x,y)\mapsto ((x+3t+3)/9,y/27)$, we obtain $\cE_t:y^2=x^3+f(t) x+g(t)$, which is an element of $\cG$.
It is straightforward to check that
\begin{equation}\label{eq:Elambda}
 \cE_\lambda,\qquad
\lambda\in\left\{t,\ \frac{1}{t},\ 1-t,\ \frac{1}{1-t},\ \frac{t-1}{t},\ \frac{t}{t-1}\right\}
\end{equation}
are all quadratic twists of $\cE_t$.
By~\cite[Proposition~III.1.7]{arithmetic}, in particular the proof of part (3), whenever $j\neq 0$ and $1728$, the six possible $\lambda$ in~\eqref{eq:Elambda} are distinct, and the map $t\mapsto j(\cE_{t})=2^8(t^2-t+1)^3/(t^2(t-1)^2)$ is six-to-one over $\overline{\Q}$. Therefore~\eqref{eq:Elambda} gives precisely the set of $\lambda$ such that $\cE_{\lambda}$ and $\cE_t$ are $\overline{\Q}$-isomorphic. By definition of $\cG$, after fixing $t$, each twist $\cE_t^d$ is counted exactly once since $d$ runs over all squarefree integers.
\end{proof}
 Then $\cE:y^2=x^3+ f(t)x+g(t)$ has three $2$-torsion points $(3(2t-1),0)$, $(3(2-t),0)$ and $(-3(t+1),0)$.

 Take $\tau=1$, so
 \[
 A(a,b)=-27(a^2-ab+b^2) \quad \text{ and }\quad
 B(a,b)=-27(2a^3-3a^2b-3ab^2+2b^3).
 \] We have
 \[
 \Delta(a,b)=-a^2b^2(a-b)^2
 \]
 and $(-3(a+b),0)$ is a $2$-torsion point of $E_{a,b}$ coming from $(-3(t+1),0)$ on $\cE$. By Velu's formula, the isogenous curve $E'_{a,b}$ gives
 \[
 \Delta'(a,b)=2^8ab(a-b)^4.
 \]
 Thus, $D_+^{(1)}(a,b)=1$, $D_+^{(2)}=a-b$, $D_-^{(1)}(a,b)=ab$, and $D_-^{(2)}=1$. So, $u_+^{(1)}=0$, $u_-^{(1)}=2$, $v_+^{(2)}=1$, and $v_-^{(2)}=0$, with $\mu=-1$ and $\sigma^2=3$.

 We could instead pick the $2$-isogeny with kernel coming from the other two $2$-torsion points of $E_{a,b}$. We collect the results in Table~\ref{table:twochoice}.
 
\subsubsection{Torsion subgroup $\Z/3\Z$}
The family of elliptic curves with torsion subgroup $\Z/3\Z$ falls under~\ref{Anormal}.
Take \[
 f(t)=3(2t+9) \quad \text{ and }\quad
 g(t)=t^2-27.
 \]
 Set $(m,\tau)=(1,3)$, so
 \[
 A(a,b)=3(2ab+9b^4) \quad \text{ and }\quad
 B(a,b)=a^2-27b^6.
 \]
 By~\cite[Lemma~5.1]{harron2017counting}, a rational elliptic curve admits a rational $3$-torsion point if and only it is $\Q$-isomorphic to $E_{a,b}$ for some $(a,b)\in \cT_{1,3}$. 
 From~\cite[Section~3]{harron2017counting}, we know that other than the exceptional curves, which form a subset of size $O(N^{-\frac{1}{12}})$ in $\cF(N)$, each $\Q$-isomorphism class appears exactly once in the parametrisation.
 We have
 \[
 \Delta(a,b)=27 (a + 5 b^3) (a + 9 b^3)^3
 \]
 and by Velu's formula $E_{a,b}$ is $3$-isogenous to an elliptic curve $E'_{a,b}$ with
 \[
 \Delta'(a,b)=3^9 (a + 5 b^3)^3 (a + 9 b^3).
 \]
 Therefore we may take $D_+(a,b)=a + 5 b^3$ and $D_-(a,b)=a + 9 b^3$, thus $u_+=u_-=1$. Notice that $6g(-9)$ is a rational square and $6g(-5)$ is not, so $v_+=1/2$ and $v_-=1$. This shows that $\mu=-1/2$ and $\sigma^2=3/2$.

 \begin{table}[!ht]
\centering
\begin{tabular}{cccccccccc}
\toprule
$T$ &$\ell$& $\deg \Delta$ & $u_+$ & $u_-$ & $v_+$ & $v_-$&$\mu$ &$\sigma^2$ & $\rho(1)$ \\
\midrule
$\Z/3\Z$ & $3$& $12$ & $1$ & $1$ & $\frac 12$ & $1$ &$-\frac 12$ & $\frac 32$& $\frac{1}{3}$\\
$\Z/5\Z$& $5$ & $ 12 $ & $ 1 $ & $ 2 $ & $ \frac12 $ & $ 2 $ & $ -\frac 32 $ & $ \frac52 $& $\frac25$ \\
$\Z/6\Z$ & $3$ & $ 12 $ & $ 2 $ & $ 2 $ & $ 1 $ & $ \frac32 $ & $ -\frac12 $ & $ \frac52 $ &$1$\\
$\Z/7\Z$ & $7$ & $ 24 $ & $ 1 $ & $ 3 $ & $ \frac12 $ & $ 3 $ & $ -\frac52 $ & $ \frac72 $& $\frac37$\\
$\Z/9\Z$ & $3$ & $ 36 $ & $ 2 $ & $ 3 $ & $ \frac32 $ & $ 3 $ & $ -\frac{3}2 $ & $ \frac92 $ &$1$\\
$\Z/10\Z$ & $5$ & $ 36 $ & $ 2 $ & $ 4 $ & $ 1 $ & $ 4 $ & $ -3 $ & $ 5 $ &$\frac45$\\
$\Z/2\Z\times \Z/6\Z$ & $3$ & $ 24 $ & $ 3 $ & $ 3 $ & $ \frac32 $ & $ 3 $ & $ -\frac32 $ & $ \frac92 $ &$1$\\
$\Z/12\Z$ & $3$ & $ 48 $ & $ 3 $ & $ 4 $ & $ \frac{5}{2} $ & $ 4 $ & $ -\frac32 $ & $ \frac{13}2 $ &$\frac73$ \\
\bottomrule\\
\end{tabular}
\caption{Families of elliptic curves with torsion subgroup $T$ with an odd prime degree $\ell$ isogeny.}
\label{table:oddprime}
\end{table}

 \begin{table}[!ht]
\centering
\begin{tabular}{ccccccccccc}
\toprule
$T$ & $\deg \Delta$ & $u_+^{(1)}$ & $u_+^{(2)}$ & $u_-^{(1)}$ & $u_-^{(2)}$&$v_+^{(2)}$ & $v_-^{(2)}$ & $\mu$ & $\sigma^2$ &$\rho(1)$\\
\midrule
$\Z/2\Z$ &$6$ & $1$ & $0$ & $1$ & $0$& $0$ & $0$ & $0$ & $2$ &$\frac{1}{2}$ \\
$\Z/4\Z$ & $12$ & $1$ & $1$ & $0$ & $1$& $\frac 12$ & $1$ & $\frac 12$ & $\frac 52$ &$1$\\
$\Z/6\Z$ & $ 12 $ & $ 2 $ & $ 0 $ & $ 2 $ & $ 0 $ & $ 0 $ & $ 0 $ & $ 0 $ & $ 4 $ &$1$\\
$\Z/8\Z$ & $ 24 $ & $ 1 $ & $ 2 $ & $ 0 $ & $ 2 $ & $ \frac32 $ & $ \frac32 $ & $ 1 $ & $ 4 $ &$\frac{7}{4}$\\
$\Z/10\Z$ & $ 36 $ & $ 3 $ & $ 0 $ & $ 3 $ & $ 0 $ & $ 0 $ & $ 0 $ & $ 0 $ & $ 6 $ &$\frac{3}{2}$\\
$\Z/12\Z$ & $ 48 $ & $ 2 $ & $ 2 $ & $ 0 $ & $ 3 $ & $ 2 $ & $ \frac52 $ & $ \frac32 $ & $ \frac{13}{2} $ &$\frac{11}{4}$\\
\bottomrule\\
\end{tabular}
\caption{Families of elliptic curves with torsion subgroup $T$ with a unique $2$-isogeny.}\label{table:twounique}
\end{table}

\begin{table}[!ht]
\centering
\resizebox{\linewidth}{!}{%
\begin{tabular}{ccccccccccccc}
\toprule
$T$ &$\substack{ x\text{-coordinate of the}\\ \text{generator of }\ker\phi}$& $\deg \Delta$ & $u_+^{(1)}$ & $u_+^{(2)}$ & $u_-^{(1)}$ & $u_-^{(2)}$&$v_+^{(2)}$ & $v_-^{(2)}$ & $\mu$ & $\sigma^2$ &$\rho(1)$&$\rho(2)$\\
\midrule
$\Z/2\Z\times \Z/2\Z$ & $-3(t+1)$& $ 6 $ & $ 0 $ & $ 1 $ & $ 2 $ & $ 0 $ & $ 1 $ & $ 0 $ & $ -1 $ & $ 3 $& $0$ &$\frac{3}{2}$\\
$\Z/2\Z\times \Z/2\Z$ & $3(2t-1)$& $ 6 $ & $ 0 $ & $ 1 $ & $ 2 $ & $ 0 $ & $ \frac12 $ & $ 0 $ & $ -\frac32 $ & $ \frac 52 $ & $-\frac{1}{2}$ &$0$\\
$\Z/2\Z\times \Z/2\Z$ & $-3(t-2)$& $ 6 $ & $ 0 $ & $ 1 $ & $ 2 $ & $ 0 $ & $ \frac12 $ & $ 0 $ & $ -\frac32 $ & $ \frac52 $ & $-\frac{1}{2}$&$0$\\
$\Z/2\Z \times \Z/4\Z$ & $-6 (t^2 + 8 t + 8)$& $ 12 $ & $ 0 $ & $ 1 $ & $ 2 $ & $ 1 $ & $ \frac12 $ & $ 1 $ & $ -\frac52 $ & $ \frac72 $ & $-1$&$-\frac{3}{4}$\\
$\Z/2\Z \times \Z/4\Z$ & $3 (t^2 + 8 t - 16)$ & $ 12 $ & $ 0 $ & $ 1 $ & $ 2 $ & $ 1 $ & $ 1 $ & $ \frac 12 $ & $ -\frac 32 $ & $ \frac 72 $ & $-\frac{1}{4}$&$\frac{9}{8}$\\
$\Z/2\Z \times \Z/4\Z$ &$3 (t^2 + 8 t + 32)$ & $ 12 $ & $ 0 $ & $ 2 $ & $ 0 $ & $ 2 $ & $ 1 $ & $ \frac 32 $ & $ -\frac{1}2 $ & $ \frac52 $ & $\frac{1}{4}$&$\frac{15}{8}$\\
$\Z/2\Z \times \Z/6\Z$& $\substack{-3 (759 t^4 - 228 t^3 - 30 t^2 + 12 t - 1)}$&$ 24 $ & $ 0 $ & $ 2 $ & $ 4 $ & $ 0 $ & $ \frac32 $ & $ 0 $ & $ -\frac52 $ & $ \frac{11}{2} $ & $-\frac{1}{2}$&$\frac{3}{2}$\\
$\Z/2\Z \times \Z/6\Z$& $\substack{6 (183 t^4 - 36 t^3 - 30 t^2 + 12 t - 1)}$ &$ 24 $ & $ 0 $ & $ 2 $ & $ 4 $ & $ 0 $ & $ \frac32 $ & $ 0 $ & $ -\frac52 $ & $ \frac{11}{2} $ & $-\frac{1}{2}$&$\frac{3}{2}$\\
$\Z/2\Z \times \Z/6\Z$& $\substack{3 (393 t^4 - 156 t^3 + 30 t^2 - 12 t + 1)}$& $ 24 $ & $ 0 $ & $ 2 $ & $ 4 $ & $ 0 $ & $ \frac32 $ & $ 0 $ & $ -\frac52 $ & $ \frac{11}{2} $ & $-\frac{1}{2}$&$\frac{3}{2}$\\
$\Z/2\Z \times \Z/8\Z$& $\substack{-6 (t^8 + 16 t^7 + 96 t^6 + 256 t^5 + 128 t^4 \\- 1024 t^3 - 3072 t^2 - 4096 t - 2048)}$ & $ 48 $ & $ 0 $ & $ 2 $ & $ 2 $ & $ 3 $ & $ 2 $ & $ \frac52 $ & $ -\frac52 $ & $ \frac{13}2 $& $-\frac{1}{4}$ &$\frac{21}{8}$\\
$\Z/2\Z \times \Z/8\Z$&$\substack{3 (t^8 + 16 t^7 + 96 t^6 + 256 t^5 - 256 t^4 \\- 4096 t^3 - 12288 t^2 - 16384 t - 8192)}$&$ 48 $ & $ 0 $ & $ 2 $ & $ 2 $ & $ 3 $ & $ 2 $ & $ \frac 52 $ & $ -\frac 52 $ & $ \frac {13}2 $ & $-\frac{1}{4}$&$\frac{21}{8}$\\
$\Z/2\Z \times \Z/8\Z$ &$\substack{3 (t^8 + 16 t^7 + 96 t^6 + 256 t^5 + 512 t^4 \\+ 2048 t^3 + 6144 t^2 + 8192 t + 4096)}$ &$ 48 $ & $ 0 $ & $ 3 $ & $ 0 $ & $ 4 $ & $ 2 $ & $ 4 $ & $ -2 $ & $ 6 $ & $0$&$3$\\
\bottomrule\\
\end{tabular}}
\caption{Families of elliptic curves with torsion subgroup $T$ with a choice of $2$-isogeny.}\label{table:twochoice}
\end{table}

\subsection{Cyclic $4$-isogeny}
We will show that the family of all elliptic curves with a cyclic $4$-isogeny is admissible under~\ref{Atwists}.
Take 
 \[
 f(t)=t^2-3
 \quad\text{ and }\quad
 g(t)=-t^2+2,
 \]
 so $\tau=1$. Then
 \[
 A(a,b)=a^2-3b^2\quad\text{ and }\quad
 B(a,b)=b(-a^2+2b^2).
 \]
 From which we can compute
 \[
 \Delta(a,b)=-a^4 (2 a - 3 b) (2 a + 3 b).
 \]
 A rational elliptic curve admits a rational cyclic isogeny of degree $4$ if and only is isomorphic over $\Q$ to $E_{a,b}$ for $(a,b)\in \cT_{2,1}$, see~\cite[Section~4]{pomerance2021elliptic}. By~\cite[Proposition~4]{pomerance2021elliptic}, up to removing a density $O(N^{-\xi})$ subfamily, there is a one-to-one correspondence between $\cG$ and the family of rational elliptic curves that admit a degree $4$ cyclic rational isogeny.
 
 The curve $E_{a,b}$ has a rational $2$-torsion point with coordinates $(b,0)$ and so, applying Velu's formula, $E_{a,b}$ is isogenous to $E_{a,b}'$ with 
 \[
 A'(a,b)=-4a^2 - 3b^2\quad\text{ and }\quad
 B'(a,b)=-2b(2a + b)(2a - b).
\]
 Then we find that
\[
 \Delta'(a,b)=64a^2(2a + 3b)^2(2a - 3b)^2,
 \]
 so we can take $D_+=(2a + b)(2a - b)$ and $D_-=a^2$. The constants from the computations are listed in Table~\ref{table:4cyclic}.

 \begin{table}[!ht]
\centering
\begin{tabular}{ccccccccccc}
\toprule
$\ell$ & $\deg \Delta$ & $u_+^{(1)}$ & $u_+^{(2)}$ & $u_-^{(1)}$ & $u_-^{(2)}$&$v_+^{(2)}$ & $v_-^{(2)}$ & $\mu$ & $\sigma^2$& $\rho(1)$ \\
\midrule
$2$ & $6$ & $2$ & $0$ & $0$ & $1$& $0$ & $\frac 12$ & $\frac32$ & $\frac 52$ & $\frac{7}{4}$\\
\bottomrule\\
\end{tabular}
\caption{Family of elliptic curves with a cyclic degree $4$ isogeny.}\label{table:4cyclic}
\end{table}

\subsection{A family with a $7$-isogeny}\label{subsec:7}
We will work with a family of elliptic curves with a $7$-isogeny and show that it is admissible under~\ref{Ah1}.
Let 
 \[
 f(t)=-3(t^{2} + 13t + 49)(t^{2} + 5t + 1)\quad\text{ and }\quad
 g(t)=2(t^{2} + 13t + 49)(t^{4} + 14t^{3} + 63t^{2} + 70t- 7)
 \]
 so $m=1$.
 In the notation of~\ref{Ah1}, $s(t)=(t^{2} + 5t + 1)^2$, so $\deg s<(24m)/(m+2)$. 
 The $j$-invariant of $\cE$ is $(t^2+13t+49)(t^2+5t+1)^3/t$. By~\cite[Table~4]{Maier}, for each $t\in \Q_{\neq 0}$ the rational elliptic curve $\cE_t$ admits a rational $7$-isogeny. Moreover, every rational elliptic curve with a rational $7$-isogeny is isomorphic over $\overline{\Q}$ to $\cE_t$ for exactly one $t\in \Q$. The fact that the $t\in \Q$ is unique follows from the fact that no rational elliptic curve admits two different rational $7$-isogenies~\cite[Proposition~2.2.6]{MolnarVoight}. 
 Then we find that
 \begin{gather*}
 A(a,b)=-3(a^{2} + 13ab + 49b^2)(a^{2} + 5ab + b),\\
 B(a,b)=2(a^{2} + 13ab + 49b^2)(a^{4} + 14a^{3}b + 63a^{2}b^2 + 70ab^3- 7b^4),
 \end{gather*}
 with
 \[
 \Delta(a,b)=-2^8\cdot 3^6 ab^7 (a^2 + 13 ab + 49ab)^2.
 \]
 By~\cite[Table~7]{Maier}, $E_{a,b}$ is $7$-isogenous to the elliptic curve $E_{a,b}'$ obtained via the involution $t\mapsto 49/t$ with
 \[
 \Delta'(a,b)=-2^8\cdot 3^6\cdot 7^6 a^7b (a^2 + 13 ab + 49ab)^2.
 \]
 We can take
 $D_+(a,b)=a$ and $D_-(a,b)=b$, and thus $u_+=u_-=1$. Moreover, $6B(1,0)=12$ and $6B(0,1)=-4116$, which are not squares in $\Q$, the splitting field of $D_+$ and $D_-$. We find that $v(D_+)=v(D_-)=1/2$, $\mu=0$, and $\sigma^2=1$.
\subsection{A family with a $13$-isogeny}\label{subsec:13}
We start with defining a curve that would be $13$-isogeneous to the family we are interested in.
Let $\cE':y^2=x^3+f'(t)x+g'(t)$ with
 \[
 f'(t)=-27(t^2 + 5t + 13)(t^2 + 6t + 13)(t^4 + 7t^3 + 20t^2 + 19t + 1),
 \]
 \[
 g'(t)=54(t^2 + 5t + 13)(t^2 + 6t + 13)^2(t^6 + 10t^5 + 46t^4 + 108t^3 + 122t^2 + 38t - 1).
 \]
 The parametrization is taken from~\cite[Table~4]{Maier}. For each $t\in \Q_{\neq 0}$, the rational elliptic curve $\cE'$ admits a $13$-isogeny $\phi':\cE'\rightarrow \cE$ over $\Q(t)$. 
 We consider the isogeny $\phi:\cE\rightarrow\cE'$ dual to $\phi'$, then $\cE$ is admissible with respect to $\phi$ under type~\ref{Ah0} with $m=2$. Recall $H_0\leq H$, so this is a family of elliptic curves with a rational $13$-isogeny and each curve in $\cF_0(N)$ has height bounded by $N$. 
 
 Since no rational elliptic curve admits two different rational $13$-isogenies, distinct $t\in\Q$ give rise to $\cE_t$ in different $\Q$-isomorphism classes as long as $\Disc(\cE_t)\neq 0$.
 Then \[\Delta(a,b)=-2^8\cdot 3^{12}ab^{13}(a^2 + 5ab + b^2)^2(a^2 + 6ab + 13b^2)^3.\]
 By~\cite[Table~7]{Maier}, each $E_{a,b}$ is $13$-isogenous to an elliptic curve with 
 \[
 \Delta'(a,b)=-2^{8}\cdot 3^{12}\cdot 13^{6} a^{13}b(a^2 + 5ab + b^2)^2(a^2 + 6ab + 13b^2)^3.
 \]
 Whence, $D_+=a$ and $D_-=b$. So, $u_+=u_-=1$, and $v_+=1/2$, $v_-=1$. We can apply Theorem~\ref{theorem:mainhom} with $\mu=1/2$ and $\sigma^2=3/2$. 

 \begin{table}[!ht]
\centering

\begin{tabular}{ccccccccc}
\toprule
$\ell$ & $\deg \Delta$ & $u_+$ & $u_-$ & $v_+$ & $v_-$&$\mu$ &$\sigma^2$&$\rho(1)$ \\
\midrule
$7$ & $12$ & $1$ & $1$ & $\frac 12$ & $\frac 12$ &$0$& $1$&$\frac{18}{7}$\\
$13$ & $24$ & $1$ & $1$ & $\frac 12$ & $1$ &$\frac 12$& $\frac 32$&$\frac{66}{13}$\\
\bottomrule\\
\end{tabular}
\caption{The subfamilies we considered of elliptic curves with a degree $\ell$ isogeny.}\label{todd}
\end{table}

\bibliographystyle{abbrv}
\bibliography{biblio}
\end{document}